\documentclass[12pt]{amsart}
\usepackage[latin1]{inputenc}
\usepackage{amsbsy}
\usepackage{amscd} 
\usepackage{amsfonts}
\usepackage{amsgen} 
\usepackage{amsmath}
\usepackage{amsopn} 
\usepackage{amssymb}
\usepackage{amstext}
\usepackage{amsthm} 
\usepackage{amsxtra}
\usepackage{rotating}
\usepackage{MnSymbol}
\usepackage{picture}

\theoremstyle{plain} 

\newtheorem{theoremintro}{Theorem}

\newtheorem{theoremcounter}{theoremcounter}[section]

\newtheorem{corollary}[theoremcounter]{Corollary}
\newtheorem{lemma}[theoremcounter]{Lemma}
\newtheorem{proposition}[theoremcounter]{Proposition}
\newtheorem{theorem}[theoremcounter]{Theorem}

\theoremstyle{definition}
\newtheorem{definition}[theoremcounter]{Definition}

\newtheorem{remark}[theoremcounter]{Remark}

\numberwithin{equation}{section}

\renewcommand{\theta}{\vartheta}
\renewcommand{\phi}{\varphi}
\renewcommand{\epsilon}{\varepsilon}
\renewcommand{\subset}{\subseteq}
\renewcommand{\supset}{\supseteq}

\newcommand{\N}{\mathbb N}
\newcommand{\NN}{\mathbb N}
\newcommand{\Z}{\mathbb Z}

\newcommand{\C}{\mathbb C}

\newcommand{\CC}{\mathcal C}
\newcommand{\cC}{\mathcal C}

\newcommand{\idpart}{|}
\newcommand{\paarpart}{\sqcap}
\newcommand{\baarpartbaustein}{\rotatebox{180}{$\sqcap$}}
\newcommand{\baarpart}{
\mathrel{\vcenter{\offinterlineskip \hbox{$\baarpartbaustein$}}}}
\newcommand{\upsubset}{\begin{rotate}{90}$\subset$\end{rotate}}
\newcommand{\downsubset}{\begin{turn}{270}$\subset$\end{turn}}
\newcommand{\singleton}{\uparrow}

\newcommand{\vierpart}{
\mathrel{\offinterlineskip
\hskip0ex\hbox{$\sqcap$}\hskip -.35ex\hbox{$\sqcap$} \hskip -0.35ex\hbox{$\sqcap$}}}

\newcommand{\vierpartrot}{
\mathrel{\vcenter{\offinterlineskip
\hbox{$\baarpart$} \vskip -.1ex \hbox{$\shortmid$} \vskip -.1ex \hbox{$\paarpart$}}}}

\newcommand{\crosspart}{
\mathrel{\offinterlineskip
\hskip 0.1ex \hbox{$/$}\hskip -.95ex\hbox{$\backslash$}}}

\newcommand{\fatcrosspart}{
\mathrel{\vcenter{\offinterlineskip
\hbox{$\baarpart\baarpart$} \vskip -0.2ex \hbox{\hskip .75ex $\times$} \vskip -.2ex \hbox{$\paarpart\paarpart$}}}}

\newcommand{\primarypart}{
\mathrel{\vcenter{\offinterlineskip
\hbox{$\baarpart$} \vskip -1.3ex \hbox{\hskip1.3ex$/$\hskip-1.2ex$-$} \vskip -1.2ex \hbox{\hskip2.2ex $\paarpart$}}}}

\newcommand{\midmid}{
\mathrel{\vcenter{\offinterlineskip
\vskip -0.2ex \hbox{$\shortmid$} \vskip -0.75ex \hbox{$\shortmid$}}}}

\newcommand{\halflibpart}{
\mathrel{\offinterlineskip
\hbox{$\bigtimes$}\hskip -1.55ex \hbox{$\midmid$} }}

\newcommand{\liegeblkn}{\begin{turn}{270}$[$\end{turn}}

\newcommand{\liegebalken}{
\mathrel{\vcenter{\offinterlineskip
\vskip -1.3ex \hbox{$\liegeblkn$}}}}

\newcommand{\longpr}{
\mathrel{\offinterlineskip
\hbox{$\shortmid$} \hskip -.7ex \hbox{$\liegebalken$} \hskip -.75ex\hbox{$\shortmid$}}}

\newcommand{\hochlongpr}{
\mathrel{\vcenter{\offinterlineskip
\vskip -.5ex \hbox{$\longpr$}}}}

\newcommand{\longpair}{
\mathrel{\offinterlineskip
\hskip0ex\hbox{$\shortmid$} \hskip -1.45ex\hbox{$\hochlongpr$} \hskip -1.4ex\hbox{$\shortmid$}}}

\newcommand{\legpart}{
\mathrel{\offinterlineskip
\hbox{$\shortmid$} \hskip -0.5ex \hbox{$\longpair$} \hskip -2.5ex \hbox{$\shortmid$} \hskip 0.8ex}}

\usepackage{calc}
\newcounter{PartitionDepth}
\newcounter{PartitionLength}

\newcommand{\partii}[3]{
 \begin{picture}(#3,#1)
 \setcounter{PartitionLength}{#3-#2}
 \setcounter{PartitionDepth}{-1-#1}
 \put(#2,\thePartitionDepth){\line(0,1){#1}}     
 \put(#3,\thePartitionDepth){\line(0,1){#1}}
 \put(#2,\thePartitionDepth){\line(1,0){\thePartitionLength}}
 \end{picture}}
\newcommand{\partiii}[4]{
 \begin{picture}(#4,#1)
 \setcounter{PartitionLength}{#4-#2}
 \setcounter{PartitionDepth}{-1-#1}
 \put(#2,\thePartitionDepth){\line(0,1){#1}}
 \put(#3,\thePartitionDepth){\line(0,1){#1}}
 \put(#4,\thePartitionDepth){\line(0,1){#1}}
 \put(#2,\thePartitionDepth){\line(1,0){\thePartitionLength}} 
 \end{picture}}

\newcommand{\upparti}[2]{
 \begin{picture}(#2,#1)
 \setcounter{PartitionDepth}{#1}
 \put(#2,0){\line(0,1){#1}}
 \end{picture}}
\newcommand{\uppartii}[3]{
 \begin{picture}(#3,#1)
 \setcounter{PartitionLength}{#3-#2}
 \setcounter{PartitionDepth}{#1}
 \put(#2,0){\line(0,1){#1}}     
 \put(#3,0){\line(0,1){#1}}
 \put(#2,\thePartitionDepth){\line(1,0){\thePartitionLength}}
 \end{picture}}
\newcommand{\uppartiii}[4]{
 \begin{picture}(#4,#1)
 \setcounter{PartitionLength}{#4-#2}
 \setcounter{PartitionDepth}{#1}
 \put(#2,0){\line(0,1){#1}}
 \put(#3,0){\line(0,1){#1}}
 \put(#4,0){\line(0,1){#1}}
 \put(#2,\thePartitionDepth){\line(1,0){\thePartitionLength}} 
 \end{picture}}
\newcommand{\uppartiv}[5]{
 \begin{picture}(#5,#1)
 \setcounter{PartitionLength}{#5-#2}
 \setcounter{PartitionDepth}{#1}
 \put(#2,0){\line(0,1){#1}}
 \put(#3,0){\line(0,1){#1}}
 \put(#4,0){\line(0,1){#1}}
 \put(#5,0){\line(0,1){#1}}
 \put(#2,\thePartitionDepth){\line(1,0){\thePartitionLength}} 
 \end{picture}}

\newcommand{\Hnfc}{H_n^\diamond}
\DeclareMathOperator{\wdepth}{wdepth}

\DeclareMathOperator{\Hom}{Hom}
\DeclareMathOperator{\mini}{min}
\DeclareMathOperator{\rl}{rl}
\DeclareMathOperator{\Span}{span}

\begin{document}
\title{The full classification of orthogonal easy quantum groups}
\author[Sven Raum and Moritz Weber]{Sven Raum$^{(1)}$ and Moritz Weber}
  \setcounter{footnote}{1}
  \footnotetext{Supported by KU Leuven BOF research grant OT/08/032 and by ANR Grant NEUMANN}
\address{Sven Raum, \'Ecole Normale Sup\'erieure de Lyon, Unit\'e de Math\'ematiques Pures et Appliqu\'ees, UMR CNRS 5669,  69364 Lyon Cedex 07, France}
\email{sven.raum@ens-lyon.fr}
\address{Moritz Weber, Saarland University, Fachbereich Mathematik, Postfach 151150,
66041 Saarbr\"ucken, Germany}
\email{weber@math.uni-sb.de}
\date{\today}

\begin{abstract}
In 1987, Woronowicz gave a definition of compact matrix quantum groups generalizing compact Lie groups $G\subset M_n(\C)$ in the   setting of noncommutative geometry.
About twenty years later, Banica and Speicher isolated a class of compact matrix quantum groups with an intrinsic combinatorial structure. These so called easy quantum groups are determined by categories of partitions. They have been proven useful in order to understand various aspects of quantum groups, in particular linked with Voiculescu's free probability theory. Furthermore, they exhibit a way to find examples of compact quantum groups besides $q$-deformations and quantum isometry groups.
These characteristics naturally motivated attempts to fully classify them. This is completed in the present article.
\end{abstract}

\maketitle

\section*{Introduction}

In the realm of noncommutative operator algebras, quantum groups are the right notion of ``symmetries''. In 1987, Woronowicz \cite{Woronowicz87} gave a definition of compact matrix quantum groups based on the theory of C$^*$-algebras generalizing compact Lie groups $G\subset M_n(\C)$.
In this setting, 
Wang \cite{Wang95, Wang98} defined quantum versions of the permutation group $S_n\subset M_n(\C)$ and the orthogonal group $O_n\subset M_n(\C)$. The idea is that the free symmetric quantum group $S_n^+$ is given by permutation matrices with noncommutative entries. Likewise, we have operator-valued orthogonal matrices in the free orthogonal quantum group $O_n^+$.
This ``liberation of groups'' to a quantum group setting was further extended and refined by Banica and Speicher in 2009 \cite{BanicaSpeicher09}. They began to study compact matrix quantum groups $G$ with $S_n\subset G\subset O_n^+$ focusing on those with an intrinsic combinatorial structure, the so called easy quantum groups or partition quantum groups. They appear in this liberation framework as very natural objects. By Woronowicz's \cite{Woronowicz88} Tannaka-Krein duality, compact matrix quantum groups in general are determined by their intertwiner spaces. In the case of  easy quantum groups, these intertwiner spaces can be given in terms of partitions of sets of a finite number of points. 

Given $k$ upper and $l$ lower points, we partition this set into several disjoint subsets (called blocks) by drawing lines connecting some of the points. These objects play a central role in Voiculescu's free probability \cite{VoiculescuDykemaNica92} due to the work of Speicher \cite{NicaSpeicher09}.

It quickly turned out that the approach by Banica and Speicher is quite powerful due to several reasons.
\begin{itemize}
 \item[(1)] Easy quantum groups give rise to many new examples (see below) of compact matrix quantum groups $S_n\subset G\subset O_n^+$ besides $q$-deformations and quantum  isometry groups. They can be understood in a very concrete way due to their combinatorics. For instance, their fusion rules may be expressed using partitions \cite{FreslonWeber13, Freslon13}. 
 \item[(2)] Easy quantum groups provide useful symmetries in the context of Voiculescu's free probability \cite{VoiculescuDykemaNica92, NicaSpeicher09}. In fact, while invariance under the group $S_n$ corresponds to (classical) independence in (classical) probability theory, invariance under the quantum group $S_n^+$ gives rise to free independence. This follows from a free de Finetti result by K\"ostler and Speicher \cite{KoestlerSpeicher09} and it has been extended by Banica, Curran and Speicher in a couple of articles \cite{BanicaCurranSpeicher12, Curran09, Curran10, Curran11, CurranSpeicher11a, CurranSpeicher11b}.
 Furthermore, like in free probability the noncrossing partitions play an important role in the theory of easy quantum  groups. For instance, the liberation from $S_n$ to $S_n^+$ (as well as from $O_n$ to $O_n^+$) is reflected by restricting to noncrossing partitions.  Hence, we have a kind of a synthesis of ideas in free probability and easy quantum groups which give natural explanations of certain phenomena.
 \item[(3)] By Woronowicz \cite{Woronowicz87}, compact matrix quantum groups admit a Haar state, analogous to the Haar measure on compact groups. This yields a canonical GNS-representation  and we obtain von Neumann algebra versions of compact matrix quantum groups. Von Neumann algebras associated with easy quantum groups have been studied by many authors, see for instance \cite{VaesVergnioux07, Brannan12a, Brannan12b, Freslon12, Isono12}.
 Other operator algebraic aspects such as $K$-theory \cite{Voigt11, VergniouxVoigt11} or C$^*$-algebraic properties \cite{Banica96, Banica97}
 have been investigated.
\end{itemize}
The work on easy quantum groups has been done by Banica, Curran, Speicher and many others. In addition to the above articles, we would like to mention \cite{BanicaCurranSpeicher10, BanicaCurranSpeicher11, Raum11, BanicaBichonCollinsCurran12, RaumWeberSimpl, RaumWeberSemiDirect, Weber13}
and related articles
\cite{BanicaVergnioux09, BanicaVergnioux10, VanDaeleWang96, BanicaSkalski11, BanicaBelinschiCapitaineCollins07, BanicaCollinsZinn-Justin09, BichonDubois-Violette13}.
Let us also mention the important preparatory work by Banica, Bichon and Collins \cite{Bichon03, BanicaCollins07, BanicaBichonCollins07}. These lists are not guaranteed to be complete, but they might help to get an overview on the subject.
For first steps to extend the easy quantum group setting from quantum subgroups of $O_n^+$ to quantum  subgroups of $U_n^+$, we refer to \cite{TarragoWeber13}.

It is desirable to have a complete list of all easy quantum groups due to the following reasons. Via the de Finetti theorems -- as well as  by their laws of characters --, they shed a light onto aspects of noncommutative distributions in free probability. Furthermore, while classifying easy quantum groups, the authors discovered a natural class of compact quantum groups (not only easy ones!) which can be completely classified and understood as semi-direct products, \cite{RaumWeberSemiDirect}. We hope that this success can be repeated for another class continuing the work of the present article.

Most of the classification process can be performed on the level of combinatorics. The question is to find all categories $\CC$ of partitions, i.e. all sets of partitions which are closed under ``intertwiner space like'' operations. In the initial article of Banica and Speicher \cite{BanicaSpeicher09} and later in \cite{Weber13} all free easy quantum groups were classified, i.e. all categories of noncrossing partitions. There are exactly seven of them, containing $S_n^+$, $O_n^+$ and the free hyperoctahedral quantum group $H_n^+$ of Banica, Bichon and Collins \cite{BanicaBichonCollins07}.
The case of those easy quantum groups which are actually groups was treated in \cite{BanicaSpeicher09}. They correspond to categories containing the partition $\crosspart\in\CC$ (which reflects that the underlying C$^*$-algebra is commutative), and there are exactly six. In \cite{BanicaSpeicher09} and \cite{BanicaCurranSpeicher10}, halfliberated quantum groups were introduced. Here, the commutativity relations $ab=ba$ on the generators of their C$^*$-algebras are replaced by $abc=cba$, corresponding to the partition $\halflibpart$. They were completely classified in \cite{Weber13}.

When aiming at a full classification of all easy quantum groups, it turns out that they split into three classes. We say that a category $\CC$ of partitions is hyperoctahedral, if the four block partition $\vierpart$ is contained in $\CC$ but the double singleton $\singleton\otimes\singleton$ is not. The non-hyperoctahedral case was treated in \cite{BanicaCurranSpeicher10} and \cite{Weber13}, whereas the class of so called group-theoretical hyperoctahedral categories has been studied in \cite{RaumWeberSimpl} and \cite{RaumWeberSemiDirect}. For the remaining case, we obtain a complete classification in the present article using the following partitions $\pi_k$  given by $k$ four blocks on $4k$ points:
\[\pi_k=a_1\ldots a_ka_k\ldots a_1a_1\ldots a_ka_k\ldots a_1\]

\begin{theoremintro}[Theorem \ref{ThMain}]
 Let $\cC$ be a non group-theoretical hyperoctahedral category. Then $\CC=\langle\pi_l,l\leq k\rangle$ for some $k\in\NN$ or $\CC=\langle\pi_l,l\in \NN\rangle$. 
\end{theoremintro}

We also write $\langle\pi_l,l\leq \infty\rangle$ for $\langle\pi_l,l\in \NN\rangle$.
We deduce the following full classification of all (orthogonal) easy quantum groups.

\begin{theoremintro}[Theorem \ref{ThMain2}]
If $G$ is an orthogonal easy quantum group, its corresponding category of partitions
\begin{itemize}
\item[(i)] either is non-hyperoctahedral (and hence it is one of the 13 cases of \cite{Weber13}),
\item[(ii)] or it is hyperoctahedral and contains $\primarypart$ and hence is group-theoretical (see \cite{RaumWeberSimpl}, \cite{RaumWeberSemiDirect}).
\item[(iii)] or it coincides with $\langle\pi_l,l\leq k\rangle$ for some $k\in\{1,2,\ldots,\infty\}$.
\end{itemize}
\end{theoremintro}

The idea of the proof is the following. Let $\CC$ be a hyperoctahedral category of partitions
with $\primarypart\notin\CC$. If $\CC$ contains a partition which looks like a ``disturbed'' version of $\pi_k$, then we already have $\pi_k\in\CC$ and hence $\langle\pi_l,l\leq k\rangle\subset \CC$. On the other hand, every partition in $\CC$ actually \emph{is} such a disturbed version -- and all these partitions may be constructed inside of $\langle\pi_l,l\leq k\rangle$. Thus, $\CC=\langle\pi_l,l\leq k\rangle$ for some $k$.

Section \ref{SectEasyQG} is devoted to give a concise introduction into easy quantum groups. In Section \ref{SectHistory} we review the history of the classification of easy quantum groups. Section \ref{SectDefWeak} provides some technical background and the information that there is no category in between $\langle\vierpart\rangle$ and $\langle\pi_2\rangle$. In Section \ref{SectStructure}, the structure of categories $\CC$ in the hyperoctahedral case with $\primarypart\notin\CC$ is studied. This provides the main tools for Section \ref{SectWeakSimpl}, where we actually prove our classification result by purely combinatorial means.
We end with a study of the C$^*$-algebras associated to the quantum groups coming from the categories $\langle\pi_l,l\leq k\rangle$.
Note that passing to the C$^*$-algebras is necessary for distinguishing the found categories of partitions.

\section*{Acknowledgements}

The authors thank Stephen Curran for discussions on categories of partitions. He pointed out the existence of the category $\langle\fatcrosspart\rangle$ which led us to the discovery of the series $\langle \pi_l,l\leq k\rangle$.

\section{Easy quantum groups}\label{SectEasyQG}

In this section, we give a complete introduction into easy quantum groups. Further details may be found in \cite{BanicaSpeicher09} or in \cite{Weber13}.

\subsection{Compact matrix quantum groups}\label{SectCMQG}

In 1987, Woronowicz (\cite{Woronowicz87} and \cite{Woronowicz91}) gave a definition of a \emph{compact matrix pseudogroup}, today also called \emph{compact matrix quantum group}.

\begin{definition}\label{DefCMQG}
 A \emph{compact matrix quantum group} (of dimension $n\in\N$) is given by
 \begin{itemize}
  \item a unital $C^*$-algebra $A$,
  \item elements $u_{ij}\in A$, $1\leq i,j\leq n$ that generate $A$ as a $C^*$-algebra,
  \item and a *-homomorphism $\Delta: A\to A\otimes_{\mini} A$, mapping $u_{ij}\mapsto \sum_{k=1}^n u_{ik}\otimes u_{kj}$.
  \item Furthermore, we require the $n\times n$-matrices $u=(u_{ij})$ and $u^t=(u_{ji})$ to be invertible.
 \end{itemize}
\end{definition}

Note that every compact matrix quantum group is a compact quantum group in the sense of \cite{Woronowicz98}. If $G\subset M_n(\C)$ is a compact group, the algebra of continuous functions $C(G)$ on $G$ can be viewed as a compact  matrix quantum group. The generators $u_{ij}$ are given by the coordinate functions evaluating the matrix entries. The map $\Delta$ is given by the above formula which dualizes matrix multiplication. More explicitly:
\begin{align*}
&\Delta: C(G)\to C(G)\otimes  C(G)\cong  C(G\times G)\\
&\Delta(f)(g,h):=f(gh), \quad \forall g,h\in G
\end{align*}
Hence the notion of a compact matrix quantum group generalizes compact matrix groups. It is therefore convenient to write $C(G)$ for the $C^*$-algebra $A$ in Definition~\ref{DefCMQG} -- even in the case when $A$ is noncommutative -- and to think of $G$ as the quantum group. Note that the object $G$ is defined in a precise way only via its associated C$^*$-algebra $C(G)$ and the comultiplication $\Delta$.

In our setting, we will treat compact matrix quantum groups which are given by universal $C^*$-algebras generated by self-adjoint elements $u_{ij}$ and the relations that $u$ and $u^t$ are orthogonal matrices plus some further relations.

If $G$ is a compact matrix quantum group of dimension $n$ and $u$ is the matrix imposed by the generators $u_{ij}$, the matrix $u^{\otimes k}\in  M_{n^k}(C(G))$ gives rise to a corepresentation of the quantum group. We associate the \emph{intertwiner spaces} to $G$, given by the collection of all sets:
\[\Hom_G( k,l)=\{T:(\C^n)^{\otimes k}\to (\C^n)^{\otimes l} \textnormal{ linear} \; | \; Tu^{\otimes k}=u^{\otimes l}T\},\quad k,l\in\N_0\]
It is an important consequence of Woronowicz's Tannaka-Krein duality (\cite{Woronowicz88}) that a compact matrix quantum group can be reconstructed from its intertwiner space (which is a tensor category). Hence, there is a duality between compact matrix quantum groups and tensor categories, which will be used in the sequel.

\subsection{Partitions}\label{SectPart}

In order to describe intertwiner spaces of quantum groups combinatorially, Banica and Speicher \cite{BanicaSpeicher09} used the concept of partitions.  A \emph{partition} $p$ is given by $k\in\N_0$ upper points and $l\in\N_0$ lower points which may be connected by lines. By this, the set of $k+l$ points is partitioned into several \emph{blocks}. We write a partition as a diagram in the following way using lines to represent the blocks:
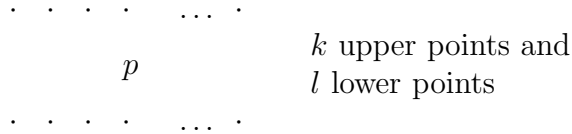
\setlength{\unitlength}{0.5cm}
\begin{center}
  \begin{picture}(14,5)    
    \put(0,1){$\cdot$}
    \put(1,1){$\cdot$}
    \put(2,1){$\cdot$}
    \put(3,1){$\cdot$}
    \put(4.5,1){$\dotsc$}
    \put(6,1){$\cdot$}

    \put(3,2.5){$p$}

    \put(0,4){$\cdot$}
    \put(1,4){$\cdot$}
    \put(2,4){$\cdot$}
    \put(3,4){$\cdot$}
    \put(4.5,4){$\dotsc$}
    \put(6,4){$\cdot$}
    
    \put(8,2.5){\begin{minipage}{3.5cm} $k$ upper points and \\ $l$ lower points \end{minipage}}
  \end{picture}
\end{center}

Three examples of such partitions are the following diagrams:
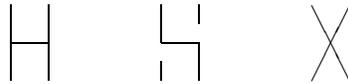
\setlength{\unitlength}{0.5cm}
\begin{center}
  \begin{picture}(10,3)
    \put(0,0){\line(0,1){2}}
    \put(1,0){\line(0,1){2}}
    \put(0,1){\line(1,0){1}}

    \put(4,0){\line(0,1){0.5}}
    \put(5,0){\line(0,1){1}}
    \put(4,2){\line(0,-1){1}}
    \put(5,2){\line(0,-1){0.5}}
    \put(4,1){\line(1,0){1}}

    \put(8,0){\line(1,2){1}}
    \put(8,2){\line(1,-2){1}}
  \end{picture}
\end{center}
In the first example, all four points are connected, and the partition consists only of one block.  In the second example, the left upper point and the right lower point are connected, whereas none of the two remaining points is connected to any other point. The third partition consists of two blocks connecting the left upper and the right lower point, and similarly the right upper and the left lower point, respectively.

The set of partitions on $k$ upper and $l$ lower points is denoted by $P(k,l)$, and the set of all partitions is denoted by $P$. If $k = l = 0$, then $P(0, 0)$ consists only of the \emph{empty partition}  $\emptyset$. 
A partition $p \in P(k,l)$ is called \emph{noncrossing}, if it can be drawn in such a way that none of its lines cross; we then write $p\in NC(k,l)$. In the above examples, the first two partitions are noncrossing whereas the third is not.

When we need to refer to points of a partition, we may label the upper points by the numbers $1, \dots, k$ from left to right and likewise $1', \dots, l'$ for the lower points.

A few partitions play a special role in this article, and they are listed here:
\begin{itemize}
 \item The \emph{singleton partition} $\singleton$ is the partition in $P(0,1)$ on a single lower point.
 \item The \emph{double singleton partition} $\singleton\otimes\singleton$ is the partition in $P(0,2)$ on two non-connected lower points.
 \item The \emph{pair partition} (also called \emph{duality partition}) $\paarpart$ is the partition in $P(0,2)$ on two connected lower points.
 \item The \emph{identity partition} (also called \emph{unit partition}) $\idpart$ is the partition in $P(1,1)$ connecting one upper with one lower point.
 \item The \emph{four block partition} $\vierpart$ is the partition in $P(0,4)$ connecting four lower points.
 \item The \emph{s-mixing partition} $h_s$ is the partition in $P(0,2s)$ for $s\in \N$ given by two blocks connecting the $2s$ points in an alternating way:
\setlength{\unitlength}{0.5cm}
\begin{center}
\begin{picture}(13,4)
\put(-1,1.5){$h_s\;=$}
\put(-0.1,1){\uppartiii{2}{1}{3}{5}}
\put(5,3){\line(1,0){1}}
\put(6.75,2.95){$\ldots$}
\put(-0.1,1){\uppartiii{1}{2}{4}{6}}
\put(6,2){\line(1,0){1}}
\put(7.75,1.95){$\ldots$}
\put(-0.1,1){\uppartii{1}{10}{12}}
\put(8.65,3){\line(1,0){0.5}}
\put(-0.1,1){\uppartii{2}{9}{11}}
\put(9.5,2){\line(1,0){1}}
\end{picture}
\end{center}
 \item The \emph{crossing partition} (also called \emph{symmetry partition}) $\crosspart$ is the partition in $P(2,2)$ connecting the upper left with the lower right point, as well as the upper right point with the lower left one. It is the partition of two crossing pair partitions.
 \item The \emph{fat crossing partition} $\fatcrosspart\in P(4,4)$ is given by two crossing  four blocks $\{1,2,3',4'\}$ and $\{3,4,1',2'\}$:
\setlength{\unitlength}{0.5cm}
 \begin{center}
 \begin{picture}(5,7)
 \put(1,5.5){1}
 \put(2,5.5){2}
 \put(3,5.5){3}
 \put(4,5.5){4} 
 \put(1,0){$1'$}
 \put(2,0){$2'$}
 \put(3,0){$3'$}
 \put(4,0){$4'$}
 \thicklines
 \put(0,6){\partii{1}{1}{2}}
 \put(0,6){\partii{1}{3}{4}}
 \put(1.8,2){\line(1,1){2}}
 \put(1.8,4){\line(1,-1){2}}
 \put(0,1){\uppartii{1}{1}{2}}
 \put(0,1){\uppartii{1}{3}{4}}
 \end{picture}
 \end{center}
 \item The \emph{half-liberating partition} $\halflibpart$ is the partition in $P(3,3)$ given by three blocks  $\{1,3'\}$, $\{2,2'\}$ and $\{3,1'\}$.
 \item The \emph{pair positioner partition} $\primarypart\in P(3,3)$ is given by a four block $\{1,2,2',3'\}$ and a pair $\{3,1'\}$.
\end{itemize}

We will also label partitions with letters and use a correspondence of partitions and words. To a partition $p\in P(0,n)$ with $k$ blocks we assign a word in the letters $a_1,\ldots, a_k$ by labelling all points of the first block by $a_1$, all points of the second block by $a_2$ and so on. Conversely, a word of length $n$ in $k$ letters gives rise to a partition $p\in P(0,n)$ given by connecting all equal letters by lines. As an example, the above partition $h_s$ can be seen as the word $w=ababab\ldots ab$ of length $2s$.
We say that $X$ is a \emph{subword} of the partition $p$, if $X$ is a word
\[X=a_{i(r)}a_{i(r+1)}\ldots a_{i(r+s)}\]
and $p$ is of the form:
\[p=a_{i(1)}a_{i(2)}\ldots a_{i(r-1)} X a_{i(r+s+1)}\ldots a_{i(n)}\]

\subsection{Operations on partitions}

There are several operations on the set $P$ of partitions:
\begin{itemize}
\item  The \emph{tensor product} of two partitions $p\in P(k, l)$ and $q\in P(k', l')$ is the partition $p\otimes q\in P(k+k', l+l')$ obtained by \emph{horizontal concatenation} (putting $p$ and $q$ side by side), i.e. the first $k$ of the $k+k'$ upper points are connected by $p$ to the first $l$ of the $l+l'$ lower points, whereas $q$ connects the remaining $k'$ upper points with the remaining $l'$ lower points.
\item The \emph{composition} of two partitions $p\in P(k, l)$ and $q\in P(l, m)$ is the partition  $qp\in P(k, m)$ obtained by \emph{vertical concatenation} (putting $p$ above $q$). Connect $k$ upper points by $p$ to $l$ middle points and then continue the lines by $q$ to $m$ lower points. This yields a partition, connecting $k$ upper points with $m$ lower points. The $l$ middle points and all lines neither connected to the $k$ upper nor to the $m$ lower points are removed.
\item The \emph{involution} of a partition $p\in P(k, l)$ is the partition $p^{*}\in P(l, k)$ obtained by turning $p$ upside down.
\item We also have a \emph{rotation} on partitions. Let $p\in P(k, l)$ be a partition connecting $k$ upper points with $l$ lower points. Shifting the very left upper point to the left of the lower points -- hereafter it still belongs to the same block as before -- gives rise to a partition in $P(k-1, l+1)$, called a \emph{rotated version} of $p$. We may also rotate the left lower point to the upper line, and we may rotate on the right-hand side of the two lines of points as well. In particular, for a partition $p\in P(0, l)$, we may rotate the very left point to the very right and vice-versa.
\end{itemize}

These operations (tensor product, composition, involution and rotation) are called the \emph{category operations}. 
A collection $\CC$ of subsets $\CC(k, l)\subseteq P(k, l)$ (for every $k, l\in\N_{0}$) is a \emph{category of partitions} if it is invariant under the category operations and if the identity partition $\idpart\in P(1, 1)$ is in $\CC$ (by rotation, we then also have $\paarpart\in\CC$). Rotation may be deduced from the other category operations if $\paarpart\in\CC$ (see \cite{BanicaSpeicher09}).
Examples of categories of partitions include the set of all partitions $P$, the set of all \emph{pair partitions} $P_{2}$ (all blocks have size two), the set of all \emph{noncrossing partitions} $NC$ (the lines may be drawn such that they do not cross) and $NC_{2}$ (\cite{BanicaSpeicher09}). We write $\CC=\langle p_1,\ldots, p_n\rangle$, if $\CC$ is the smallest category of partitions containing the partitions $p_1,\ldots, p_n$. We say that $\CC$ is \emph{generated} by $p_1,\ldots,p_n$.

We will often speak of ``\emph{applying the pair partition}'' to a partition $p\in\CC$ or simply ``\emph{using the pair partition}''. By this, we mean that we compose $p$ with the partition $\idpart^{\otimes\alpha}\otimes\baarpart\otimes\idpart^{\otimes\beta}$ for suitable $\alpha$ and $\beta$ (maybe iteratively) -- an operation which can be done inside the category, i.e. the partition resulting from this procedure is in $\CC$ again. This ``erases'' or ``couples'' two points of $p$. As an example consider the partitions $\pi_2\in P(0,8)$ and $\pi_3\in P(0,12)$, written as words:
\[\pi_2=abbaabba,\qquad \pi_3=abccbaabccba\]
Composing $\pi_3$ with $\idpart^{\otimes 2}\otimes\baarpart\otimes \idpart^{\otimes 4}\otimes\baarpart\otimes \idpart^{\otimes 2}$ yields that $\pi_2\in\langle\pi_3\rangle$.
Similarly we see that $\langle\vierpart\rangle\subset\langle\fatcrosspart\rangle$.

\subsection{Linear maps associated to partitions}

Given a partition $p \in P(k,l)$ and two multi-indices $(i(1), \dotsc, i(k))$, $(j(1), \dotsc, j(l))$, we can label the diagram of $p$ with these numbers (the upper and the lower row both are labelled from left to right, respectively) and we put
\[\delta_p(i,j)
  =
  \begin{cases}
    1 & \text{if } p \text{ connects only equal indices,} \\
    0 & \text{otherwise}
  \end{cases}\]
For every $n \in \NN$, we fix a basis $e_1,\ldots, e_n$ of $\C^n$ and define a map $T_p: (\C^n)^{\otimes k} \to (\C^n)^{\otimes l}$  associated with $p$ by:
\[T_p(e_{i(1)} \otimes \dotsm \otimes e_{i(k)}) = \sum_{1 \leq j(1), \dotsc, j(l) \leq n} \delta_p(i, j) \cdot e_{j(1)} \otimes \dotsm \otimes e_{j(l)} \]

By \cite{BanicaSpeicher09}, we know that the category operations on partitions may be translated to the maps $T_p$ in the following way:
\begin{itemize}
 \item $T_{p\otimes q}=T_p\otimes T_q$
 \item $T_{qp}=n^{-\rl(p,q)}T_qT_p$, where $\rl(p,q)$ denotes the number of removed loops when composing $p$ and $q$.
 \item $T_{p^*}=(T_p)^*$
\end{itemize}

The maps $T_p$ can be normalized in such a way that they become partial isometries, see \cite{FreslonWeber13}.

\subsection{Definition of easy quantum groups}\label{SubsectDefEasy}

In the 1990's, S. Wang (\cite{Wang95}, \cite{Wang98}) introduced two important examples of compact matrix quantum groups: the \emph{free orthogonal quantum group} $O_n^+$ and the \emph{free symmetric quantum group} $S_n^+$. For $n\in\N$, they are given by the following universal C$^*$-algebras (endowed with a comultiplication as in Definition \ref{DefCMQG}):
\begin{align*}
A_{o}(n) &= C^*\big(u_{ij}, 1\leqslant i, j\leqslant n\;\big|\; u_{ij}^{*} = u_{ij}, 
\sum_k u_{ik}u_{jk} = \sum_k u_{k i}u_{k j} = \delta_{ij}\big) \\
A_{s}(n) & = C^*\big(u_{ij}, 1\leqslant i, j\leqslant n \;\big|\; u_{i j}^{*} = u_{i j} = u_{i j}^{2},  \sum_k u_{i k} = \sum_k u_{k j} = 1\big)
\end{align*}
Note that if we add the relations $u_{ij}u_{kl} = u_{kl}u_{ij}$ to $A_{o}(n)$ (resp. $A_{s}(n)$), the resulting quotient C$^*$-algebra is isomorphic to the algebra of continuous functions $C(O_{N})$ (resp. $C(S_{N})$). Hence, $O_n^+$ is a ``liberated version'' of $O_n$, likewise for $S_n^+$ and $S_n$.

Using the linear maps $T_p$ indexed by partitions $p\in P$, a basis of the intertwiner spaces of $O_n^+$ and $S_n^+$ can be given explicitly.
\begin{align*}
\Hom_{O_n^+}(k, l) & = \Span\{T_{p} \;|\; p\in NC_{2}(k, l)\} \\
\Hom_{S_n^+}(k, l) & = \Span\{T_{p} \;|\; p\in NC(k, l)\}
\end{align*}
Viewing the groups $O_n$ and $S_n$ as quantum groups (see Section \ref{SectCMQG}), their intertwiner spaces are given by:
\begin{align*}
\Hom_{O_n}(k, l) & = \Span\{T_{p} \;|\; p\in P_{2}(k, l)\} \\
\Hom_{S_n}(k, l) & = \Span\{T_{p} \;|\; p\in P(k, l)\}
\end{align*}
Hence the liberation procedure is given by passing from all partitions to the noncrossing ones.

By functoriality of the intertwiner spaces, we know that for any compact matrix quantum group $S_n\subset G \subset O_n^+$, its intertwiner spaces fulfill:
\[\Span\{T_{p} \;|\; p\in P(k, l)\}\supset \Hom_G(k, l)\supset  \Span\{T_{p} \;|\; p\in NC_{2}(k, l)\}\]

In 2009, Banica and Speicher came up with the following natural definition \cite[Def 6.1]{BanicaSpeicher09} refining the liberation of compact Lie groups but also going far beyond it.

\begin{definition}
A compact matrix quantum group $S_n\subset G\subset O_n^+$ is called \emph{easy} (or: \emph{partition quantum group}) if
there is a category of partitions $\CC$ such that:
\[\Hom_G(k, l) = \Span\{T_{p} \;|\; p\in \CC(k, l)\},\qquad \textnormal{for all } k, l\in\N_0\]
\end{definition}

In principle, we do not need to require that $\CC$ is a category of partitions. If the spaces $\Hom_G(k,l)$ are spanned by maps $T_p$ indexed by some arbitrary subsets of $P(k,l)$, then we can find a unique category of partitions such that the spaces $\Hom_G(k,l)$ are given by this category in the above sense. This follows automatically from the fact that the intertwiner spaces form a tensor category and that the relations $T_{p\otimes q}=T_p\otimes T_q$, $T_{pq}=n^{-b(p,q)}T_pT_q$ and $T_{p^*}=T_p^*$ hold (see \cite{BanicaSpeicher09}).
Examples of easy quantum groups include $S_n, O_n, S_n^+$ and $O_n^+$. By Woronowicz's Tannaka-Krein duality, categories of partitions and intermediate easy quantum groups in between $S_n$ and $O_n^+$ are in one-to-one correspondence; to be more precise, to each category of partitions we assign a unique family $(G_n)_{n\in\N}$ of easy quantum groups, namely one quantum group $G_n$ for every dimension $n\in\N$. 
Note that small dimensions might not distinguish between different categories. For instance, for $n=3$, the quantum groups $S_3$ and $S_3^+$ coincide.
By definition, easy quantum groups carry a lot of combinatorial data. We refer to \cite{BanicaSpeicher09}, \cite{BanicaCurranSpeicher10} and \cite{Weber13} for further details concerning easy quantum groups. 

The above definition only covers \emph{orthogonal} easy quantum groups. Wang also defined a liberation $U_n^+$ of the unitary group $U_n$. A notion of \emph{unitary} easy quantum groups $S_n\subset G\subset U_n^+$ is defined in \cite{TarragoWeber13}. In the present article we only consider orthogonal easy quantum groups.

\section{The history of the classification of easy quantum groups}\label{SectHistory}

In this section we give a brief overview on the classification of orthogonal easy quantum groups following the chronological order. It began in \cite{BanicaSpeicher09}, and was continued in \cite{BanicaCurranSpeicher10}, \cite{Weber13},\cite{RaumWeberSimpl} and \cite{RaumWeberSemiDirect}. In Section \ref{SectWeakSimpl} we complete the classification.

\subsection{The free case}

Since the liberations of $O_n$ to $O_n^+$ and of $S_n$ to $S_n^+$ are the first examples of easy quantum groups, it is natural to consider categories of noncrossing partitions first.
By \cite[Theorem 3.16]{BanicaSpeicher09} and \cite[Corollary 2.10]{Weber13}, there are exactly seven \emph{free easy quantum groups} (also called \emph{free orthogonal quantum groups}), namely:
\begin{align*}
 B_n^+ &\qquad\subset &{B_n'}^{\!+} &\qquad\subset &B_n^{\#+} &\qquad\subset &O_n^+\\
\upsubset\; & &\upsubset\; & & & &\upsubset\;\\
 S_n^+ &\qquad\subset &{S_n'}^{\!+} & &\subset &  &H_n^+ &.
\end{align*}

The corresponding seven categories of partitions are described as follows. They are subclasses of $NC$.
\begin{align*}
 \langle\singleton\rangle &\qquad\supset &\langle\legpart\rangle &\qquad\supset &\langle\singleton\otimes\singleton\rangle &\qquad\supset &\langle\emptyset\rangle=NC_2\\
\downsubset\; & &\downsubset\; & & & &\downsubset\;\\
 \langle\singleton, \vierpart\rangle=NC &\qquad\supset &\langle\singleton\otimes\singleton, \vierpart\rangle & &\supset &  &\langle\vierpart\rangle & .
\end{align*}

\subsection{The group case}

Besides the noncrossing categories, there are many categories which contain partitions that have some crossing lines. The most prominent partition which involves a crossing is the crossing partition $\crosspart\in P(2,2)$. Every category containing the crossing partition corresponds to a group. By \cite[Theorem~2.8]{BanicaSpeicher09} we know that there are exactly six easy groups.
\begin{align*}
 B_n &\qquad\subset &B_n'  &\qquad\subset &O_n\\
\upsubset\; & &\upsubset\; &  &\upsubset\;\\
 S_n &\qquad\subset &S_n' &\qquad\subset  &H_n&.
\end{align*}

Accordingly, there are exactly six categories of partitions containing the crossing partition $\crosspart$.

\begin{align*}
 \langle\crosspart,\singleton\rangle  &\qquad\supset &\langle\crosspart,\singleton\otimes\singleton\rangle &\qquad\supset &\langle\crosspart\rangle=P_2\\
\downsubset\; & &\downsubset\; &  &\downsubset\;\\
 \langle\crosspart,\singleton, \vierpart\rangle=P &\qquad\supset &\langle\crosspart,\singleton\otimes\singleton, \vierpart\rangle &\qquad\supset  &\langle\crosspart,\vierpart\rangle &.
\end{align*}

Note that  the two categories $\langle\crosspart,\legpart\rangle$ and $\langle\crosspart,\singleton\otimes\singleton\rangle$ coincide. 

\subsection{The halfliberated case}

\emph{Half-liberated} easy quantum groups were introduced in \cite{BanicaSpeicher09} and \cite{BanicaCurranSpeicher10}. They arise as a kind of a partial liberation -- the commutativity relations $u_{ij}u_{kl}=u_{kl}u_{ij}$ are replaced by the weaker relation $u_{ij}u_{kl}u_{st}=u_{st}u_{kl}u_{ij}$ (cf. Subsection \ref{SubsectDefEasy}).
The corresponding categories contain the half-liberating partition $\halflibpart$ but not the crossing partition $\crosspart$.
By \cite[Theorem 4.13]{Weber13}, there are exactly the following half-liberated easy quantum groups, containing the \emph{hyperoctahedral series} $H_n^{(s)}$, $s\geq 3$ of \cite[Definition 3.1]{BanicaCurranSpeicher10}.

\begin{align*}
 B_n^{\#*} &&\qquad\subset &&O_n^*\\
 &&\quad &&\upsubset\;\\
 &&  &&H_n^*\\
 && &&\upsubset\;\\
 &&  &&H_n^{(s)}, s\geq 3&.
\end{align*}

The corresponding  categories of partitions are described as follows.

\begin{align*}
 \langle\halflibpart,\singleton\otimes\singleton\rangle &&\qquad\supset &&\langle\halflibpart\rangle\\
 && &&\downsubset\;\\
 && &&\langle\halflibpart, \vierpart\rangle\\
 && &&\downsubset\;\\
 && &&\langle\halflibpart, \vierpart, h_s\rangle&.
\end{align*}

\subsection{The non-hyperoctahedral case}

To begin with the full classification of all orthogonal easy quantum groups, we recall some facts from \cite[Corollary 3.7, Lemma 3.9, Lemma 3.10]{Weber13}. Let $\CC$ be a category of partitions.
\begin{itemize}
 \item If $\vierpart\notin\CC$, then all blocks of partitions $p\in\CC$ have length at most two.
 \item If $\vierpart\notin\CC$ and $\CC\not\subset NC$, then $\crosspart\in\CC$ or $\halflibpart\in\CC$.
 \item If $\vierpart\in\CC$, $\singleton\otimes\singleton\in\CC$ and $\CC\not\subset NC$, then $\crosspart\in\CC$.
\end{itemize}
Using the results from the previous subsections, we infer that there are exactly nine categories $\CC$ such that $\vierpart\notin\CC$ (four noncrossing ones, three containing $\crosspart$, two containing $\halflibpart$ but not $\crosspart$). If $\vierpart\in\CC$ and $\singleton\otimes\singleton\in\CC$, then there are four cases (two noncrossing, two containing $\crosspart$).

\begin{definition}
 A category of partitions $\mathcal C\subset P$ is called \emph{hyperoctahedral}, if $\vierpart\in\mathcal C$ but $\singleton\otimes\singleton\notin\mathcal C$. Otherwise, it is called \emph{non-hyperoctahedral}.
\end{definition}

By the above considerations (see also \cite{BanicaCurranSpeicher10} and \cite[Theorem 3.12]{Weber13}), we know that there are exactly 13 non-hyperoctahedral categories of partitions.
As for the case of hyperoctahedral categories the situation is more complicated.

\subsection{The group-theoretical hyperoctahedral case}

An easy quantum group is called hyperoctahedral, if the corresponding category of partitions is hyperoctahedral.
An alternative definition is the following. A quantum group $S_n\subset G\subset O_n^+$ is hyperoctahedral, if $G$ is a quantum subgroup of $H_n^+$ (this is equivalent to $\vierpart\in\CC$ if $G$ is easy) and the fundamental representation $u$ of $G$ is irreducible (equivalent to $\singleton \otimes\singleton\notin \mathcal C$ for easy quantum groups). This definition has the advantage that we do not need any easiness assumption.

In \cite{RaumWeberSimpl} and \cite{RaumWeberSemiDirect}, we investigated hyperoctahedral categories of a special form, which we review here.

\begin{definition}
 A category $\mathcal C\subset P$ of partitions is called \emph{group-theoretical}, if $\primarypart\in\mathcal C$.
\end{definition}

Note that most of the group-theoretical categories are hyperoctahedral (applying the pair partition to $\primarypart$, we immediately see that group-theoretical categories contain the four block  partition $\vierpart$). The only two exceptions (i.e. group-theoretical categories containing $\singleton\otimes\singleton$) are $\langle\crosspart,\vierpart,\singleton\otimes\singleton\rangle$
and $\langle\crosspart,\vierpart,\singleton\rangle$.

The structure of group-theoretical quantum groups is quite algebraic: To a group-theoretical category $\CC$, we associate a certain subgroup $F(\CC)$ of the infinite free product of the cyclic group $\Z_2$. This subgroup contains the full information about $\CC$, since the category is completely determined by its subset of those partitions which look like elements in $\Z_2^{*\infty}$. Recall from Subsection \ref{SectPart} that we may view partitions as words.

\begin{definition}[{\cite{RaumWeberSimpl}}]
 A partition $p\in P$ is in \emph{single leg form}, if $p$ is -- as a word -- of the form $p=a_{i(1)}a_{i(2)}\ldots a_{i(n)}$, 
where $a_{i(j)}\neq a_ {i(j+1)}$ for $j=1,\ldots,n-1$.
The letters $a_1,\ldots,a_k$ correspond to the points connected by the partition $p$.
In other words, in a partition in single leg form no two consecutive points belong to the same block.
\end{definition}

A group-theoretical category $\CC$ is generated by its partitions in single leg form and the partition $\primarypart$.
The property $\primarypart\in\mathcal C$ translates to the fact that the generators $u_{ij}$ of a group-theoretical easy quantum group $G\subset O_n^+$ are partial isometries fulfilling  $u_{ij}u_{kl}^2=u_{kl}^2u_{ij}$, i.e. the squares of the generators are central projections. 

In \cite{RaumWeberSemiDirect} and \cite{RaumWeberSimpl} we prove that the map $\CC\mapsto F(\CC)$ is a lattice isomorphism between group-theoretical categories and a large class of subgroups of $\Z_2^{*\infty}$ and we deduce that there are uncountably many pairwise non-isomorphic easy quantum groups. This proves that the class of easy quantum groups is quite rich. Moreover, the map $F$ translates the problem of classifying group-theoretical categories to a problem in group theory, a problem in fact which is far from being solvable. In \cite{RaumWeberSemiDirect} we show that it contains the problem of understanding all varieties of groups.
Nevertheless, the map $F$ helps to explain the structure of group-theoretical easy quantum groups in a satisfying way. Denote by $F_n(\CC)$ the restriction of $F(\CC)$ to $\Z_2^{*n}$.

\begin{theorem}[{\cite{RaumWeberSemiDirect}}]
 Let $S_n\subset G\subset O_n^+$ be an easy quantum group with associated category $\CC$ of partitions. If $\CC$ is group-theoretical, then $G$ may be written as a semi-direct product:
 \[C(G)\cong C^*(\Z_2^{*n}/F_n(\mathcal C)) \Join C(S_n)\]
\end{theorem}

This decomposition contains a lot of information. See \cite{RaumWeberSemiDirect} for consequences of this picture.


\section{The fat crossing partition in hyperoctahedral categories}\label{SectDefWeak}

Let us now consider the remaining cases in the classification of orthogonal easy quantum groups. 
In this section, the fat crossing partition $\fatcrosspart$ will play a special role. 

\begin{lemma}\label{LemBezBasePart}
 The following partitions may be generated inside the following categories using the category operations.
\begin{itemize}
 \item[(i)] $\vierpart\in\langle\fatcrosspart\rangle$
 \item[(ii)] $\fatcrosspart\in\langle\primarypart\rangle$
 \item[(iii)] $\primarypart\in\langle h_s\rangle$ for all $s\geq 3$
\end{itemize}
\end{lemma}
\begin{proof}
(i) Simply apply the pair partition twice to $\fatcrosspart$.

(ii) Compose the tensor product $\vierpart\otimes\vierpart$ with $\primarypart$ in the following way:
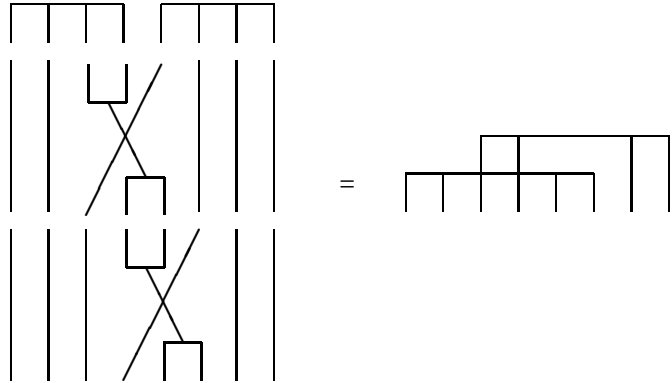
\newsavebox{\primary}
\savebox{\primary}
{ \begin{picture}(5,6)
 \thicklines
 \put(0,6){\partii{1}{1}{2}}
 \put(1.2,1){\line(1,2){2}}
 \put(1.8,4){\line(1,-2){1}}
 \put(0,1){\uppartii{1}{2}{3}}
 \end{picture}}

\setlength{\unitlength}{0.5cm}
\begin{center}
\begin{picture}(19,12)
\put(0,10){\uppartiv{1}{1}{2}{3}{4}}
\put(0,10){\uppartiv{1}{5}{6}{7}{8}}
\put(0,5.5){\upparti{4}{1}}
\put(0,5.5){\upparti{4}{2}}
\put(1.8,4.4){\usebox{\primary}}
\put(0,5.5){\upparti{4}{6}}
\put(0,5.5){\upparti{4}{7}}
\put(0,5.5){\upparti{4}{8}}
\put(0,1){\upparti{4}{1}}
\put(0,1){\upparti{4}{2}}
\put(0,1){\upparti{4}{3}}
\put(2.8,0){\usebox{\primary}}
\put(0,1){\upparti{4}{7}}
\put(0,1){\upparti{4}{8}}
\put(10,6){$=$}
\put(10.5,5.5){\uppartiv{1}{1}{2}{5}{6}}
\put(10.5,5.5){\uppartiv{2}{3}{4}{7}{8}}
\end{picture}
\end{center}
Using rotation, we infer $\fatcrosspart\in\langle\primarypart\rangle$.

(iii) We construct a rotated version  of $\primarypart$ using $h_s\otimes \idpart^{\otimes 3}$ and its rotated version:
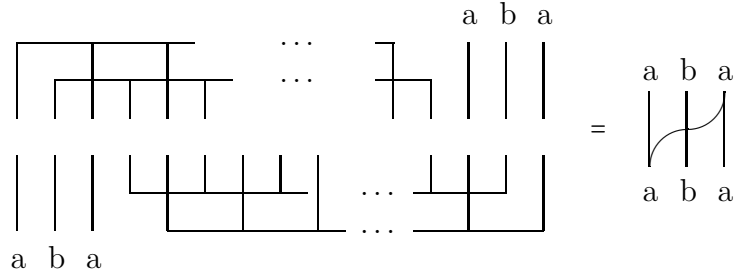
\setlength{\unitlength}{0.5cm}
\begin{center}
\begin{picture}(21,9)
\put(0,5){\uppartiii{2}{1}{3}{5}}
\put(5,7){\line(1,0){1}}
\put(8.25,6.95){$\ldots$}
\put(0,5){\uppartiii{1}{2}{4}{6}}
\put(6,6){\line(1,0){1}}
\put(8.25,5.95){$\ldots$}
\put(12,5){\upparti{2}{1}}
\put(12,5){\upparti{2}{2}}
\put(12,5){\upparti{2}{3}}
\put(10.8,7){\line(1,0){.5}}
\put(10.8,6){\line(1,0){1.5}}
\put(10,5){\upparti{2}{1}}
\put(11,5){\upparti{1}{1}}

\put(3,5){\partiii{1}{1}{3}{5}}
\put(8,3){\line(1,0){1}}
\put(10.4,2.95){$\ldots$}
\put(3,5){\partiii{2}{2}{4}{6}}
\put(9,2){\line(1,0){1}}
\put(10.4,1.95){$\ldots$}
\put(0,2){\upparti{2}{1}}
\put(0,2){\upparti{2}{2}}
\put(0,2){\upparti{2}{3}}
\put(11.8,3){\line(1,0){.5}}
\put(11.8,2){\line(1,0){1.5}}
\put(11,5){\partii{2}{2}{4}}
\put(11,5){\partii{1}{1}{3}}

\put(16.5,4.5){$=$}
\put(16.8,3.7){\upparti{2}{1}}
\put(16.8,3.7){\upparti{2}{2}}
\put(16.8,3.7){\upparti{2}{3}}
\put(19.1,3.7){\oval(2,2)[tl]}
\put(19.1,5.7){\oval(2,2)[br]}

\put(1.1,1){a}
\put(2.1,1){b}
\put(3.1,1){a}
\put(13.1,7.5){a}
\put(14.1,7.5){b}
\put(15.1,7.5){a}

\put(17.9,2.8){a}
\put(18.9,2.8){b}
\put(19.9,2.8){a}
\put(17.9,6){a}
\put(18.9,6){b}
\put(19.9,6){a}
\end{picture}
\end{center}
\end{proof}

Analoguous to our article on group-theoretical categories \cite{RaumWeberSimpl} we can restrict to partitions of a special form. This is developed in the sequel.

\begin{definition}\label{DefSingleDouble}
Let $p\in P(0,n)$ be a partition with $k$ blocks and let the letters $a_1,\ldots,a_k$ correspond to these blocks.
Then $p$ is -- as a word -- of the form
\[p=a_{i(1)}^{k(1)}a_{i(2)}^{k(2)}\ldots a_{i(m)}^{k(m)}\]
for some $k(1),\ldots,k(m)\in\NN$.
We say that $p$ is in \emph{single-double leg form}, if  $a_{i(j)}\neq a_ {i(j+1)}$ for $j=1,\ldots,m-1$ and $k(j)\in\{1,2\}$ for all $j$.
We call $a_{i(j)}^{k(j)}$ a \emph{single leg} if $k(j)=1$, and  a \emph{double leg} if $k(j)=2$.
\end{definition}

In hyperoctahedral categories, the partitions in single-double leg form determine the category completely, as can be seen in the following.

\begin{lemma}\label{LemBlockVerbinden}
 Let $\cC$ be a category of partitions containing the four block partition $\vierpart$, and let $p\in\cC$.
In $\cC$ we can \emph{connect neighboring blocks} of $p$, i.e. the partition $p'$ obtained from $p$ by combining two blocks of $p$ which have at least two neighbouring points is again in $\cC$.
\end{lemma}
\begin{proof}
By rotation, we may assume that $p$ has no upper points. Composing $p$ with $\idpart^{\otimes \alpha}\otimes \vierpartrot\otimes\idpart^{\otimes\beta}$ for suitable $\alpha$ and $\beta$ yields the result.
\end{proof}

\begin{lemma}\label{LemReduzieren1}
 Let $\cC$ be a category of partitions. Let $p\in P(0,n)$ be a partition, seen as the word $p=a_{i(1)}^{k(1)}a_{i(2)}^{k(2)}\ldots a_{i(m)}^{k(m)}$.
 We put $k(j)':= \begin{cases} 1 &\textnormal{if $k(j)$ is odd}\\ 2 &\textnormal{if $k(j)$ is even}\end{cases}$, and  $p':=a_{i(1)}^{k(1)'}a_{i(2)}^{k(2)'}\ldots a_{i(m)}^{k(m)'}$.
 
 If $\cC$ contains the four block partition $\vierpart$, then $p\in\mathcal C$ if and only if $p'\in\mathcal C$. 
\end{lemma}
\begin{proof}
Use the pair partition to infer that $p'\in\cC$ whenever $p\in\cC$. Conversely, compose $p'$ iteratively with partitions $\idpart^{\otimes \alpha}\otimes \paarpart\otimes\idpart^{\otimes\beta}$  and use Lemma \ref{LemBlockVerbinden} to increase $k(j)'$ by two, eventually obtaining $p$. 
\end{proof}

If $p\in P(0,n)$ is a partition, we can associate a partition $p'\in P$ in single-double leg form to it, obtained by applying the above procedure iteratively. Note that $p'$ might be the empty partition $\emptyset\in P(0,0)$.

\begin{proposition}\label{LemReduzieren2}
 Let $\cC$ be a category of partitions containing the four block partition $\vierpart$ and let $p\in P$. 
Then $p\in\mathcal C$ if and only if its associated partition in single-double leg form is in $\cC$. 
\end{proposition}
\begin{proof}
Apply Lemma \ref{LemReduzieren1} iteratively. 
\end{proof}

The preceding lemma allows us to restrict our study to partitions in single-double leg form. 
We now show that the category $\langle\fatcrosspart\rangle$ is a base case for all hyperoctahedral categories that contain at least one crossing partition.

\begin{lemma}\label{RemSizeTwo}
Let $\CC$ be a category of partitions such that $\singleton\otimes\singleton\notin\CC$ and let $p\in\CC$. Then $p$ is of even length and every block has size at least two.
\end{lemma}
\begin{proof}
If $p\in P(0,2k+1)$ is of odd length, we compose it with $\idpart\otimes\baarpart^{\otimes k}$ in order to obtain $\singleton\in\CC$, in contradiction to $\singleton\otimes\singleton\notin\CC$. If $p$ contains a block of length one, it is of the form $p=\singleton\otimes p'$ after rotation. We use the pair partition to infer $\singleton\otimes\singleton\in\CC$.
\end{proof}

\begin{proposition}\label{PropFat}
 Let $\mathcal C$ be a hyperoctahedral category of partitions with $\mathcal C\neq\langle\vierpart\rangle$. Then the fat crossing partition $\fatcrosspart$ is in $\mathcal C$.
\end{proposition}
\begin{proof}
We show that one of the following cases hold for $\cC$:
\begin{itemize}
 \item $\fatcrosspart\in\mathcal C$.
 \item $\primarypart\in\mathcal C$.
 \item $h_s\in\mathcal C$ for some $s\geq 3$.
\end{itemize}
By Lemma \ref{LemBezBasePart} this will complete the proof.

The only hyperoctahedral category of \emph{noncrossing} partitions is $\langle\vierpart\rangle$. Thus, $\mathcal C$ contains a partition $p\in P\backslash NC$ with a crossing. We may assume that $p$ consists only of two blocks (connect all other blocks with one of the crossing blocks, using Lemma~\ref{LemBlockVerbinden}). Furthermore, we may assume that no three points in a row are connected by one of the blocks, by using the pair partition. Hence, we may write $p$ as
\[p=a^{k_1}b^{k_2}a^{k_3}b^{k_4}\ldots a^{k_n} \qquad \textnormal{or} \qquad p=a^{k_1}b^{k_2}a^{k_3}b^{k_4}\ldots b^{k_n}\]
where $k_i\in\{1,2\}$ and $n\geq 4$. Note that the length of $p$ is even by Lemma \ref{RemSizeTwo}.

If all $k_i=1$, then $p=h_s$ for some $s\geq 2$. If $s=2$, we have $\crosspart\in\CC$ and hence $\fatcrosspart\in\CC$. Otherwise, we may assume $k_1=2$ by rotation. If $n\geq 5$, we may erase the two points $a^{k_1}$ using the pair partition and we obtain a partition $p'\in\mathcal C$ which still has a crossing. Iterating this procedure, we either end up with a partition $h_s$ for some $s\geq 2$ or with a partition $p\in\mathcal C$ such that $k_1=2$ and $n=4$. In the latter case, $p$ is of length six or eight. There are exactly four cases of such a partition:
\begin{itemize}
 \item $p=aababb$ -- An application of the pair partition would yield $\singleton\otimes\singleton\in\mathcal C$ which is a contradiction.
 \item $p=aabaab$ -- This is a rotated version of $\primarypart$.
 \item $p=aabbab$ -- Again this would yield $\singleton\otimes\singleton\in\mathcal C$.
 \item $p=aabbaabb$ -- This is $\fatcrosspart$ in a rotated version.
\end{itemize}
\end{proof}

\begin{corollary}
If $\CC$ is a hyperoctahedral category of partitions, then we are in  exactly one of the following cases.
\begin{itemize}
 \item[(i)] $\CC=\langle\vierpart\rangle$
 \item[(ii)] $\CC$ contains $\fatcrosspart$ but not $\primarypart$.
 \item[(iii)] $\CC$ is group-theoretical.
\end{itemize}
\end{corollary}

Hence, the only open case in the classification of easy quantum groups is the one of hyperoctahedral categories containing $\fatcrosspart$ but not $\primarypart$.

\section{Allowed and forbidden words and the numbers $\wdepth$}\label{SectStructure}

Recall that when studying categories of partitions, we can always restrict to partitions $p\in P(0,n)$, i.e. to partitions which have no upper points, using rotation. These partitions can be written as words in letters $a_1,\ldots a_k$ corresponding to the $k$ blocks of $p$ (see Section 1.2). In the sequel, we will often identify partitions and words. 
The following partitions play a crucial role in the classification of non group-theoretical categories of partitions.

\begin{definition}
For $k\in\NN$ we define the partition $\pi_k\in P(0,4k)$ by:
\[\pi_k:=a_1a_2\ldots a_ka_k\ldots a_2a_1a_1a_2\ldots a_ka_k\ldots a_2a_1\]
Here, the letters $a_i$ are mutually different, hence $\pi_k$ consists of $k$ different blocks, each of size four.
Note, that $\pi_1$ coincides with $\vierpart$ whereas $\pi_2=a_1a_2a_2a_1a_1a_2a_2a_1$ is a rotated version of $\fatcrosspart$.
\end{definition}

Using the pair partition, we infer that $\pi_l\in\langle \pi_k\rangle$ for all $l\leq k$. Hence, we have $\langle\pi_k\rangle=\langle\pi_l,l\leq k\rangle$. Furthermore, $\pi_k\in\langle \primarypart \rangle$ for all $k\in\N$. Indeed, assuming $\pi_k\in\langle \primarypart \rangle$, we infer that $\vierpart\otimes\pi_k$ is in $\langle \primarypart \rangle$, where we label $\vierpart=a_{k+1}a_{k+1}a_{k+1}a_{k+1}$. Composing $\vierpart\otimes\pi_k$ iteratively with partitions $\idpart^{\otimes \alpha}\otimes\primarypart\otimes\idpart^{\otimes\beta}$, we can shift pairs $a_{k+1}a_{k+1}$ in between copies of $a_k$ in order to obtain $\pi_{k+1}\in\langle\primarypart\rangle$. As an example, consider $\vierpart\otimes\pi_2$ seen as the word:
\[a_3a_3a_3a_3a_1a_2a_2a_1a_1a_2a_2a_1\]
Composing this partition with $\idpart^{\otimes 2}\otimes\primarypart\otimes\idpart^{\otimes 7}$ yields:
\[a_3a_3a_1a_3a_3a_2a_2a_1a_1a_2a_2a_1\]
Composition with $\idpart^{\otimes 3}\otimes\primarypart\otimes\idpart^{\otimes 6}$ yields:
\[a_3a_3a_1a_2a_3a_3a_2a_1a_1a_2a_2a_1\]
Now, shifting the first pair $a_3a_3$ in the same way, we obtain $\pi_3$.

We infer that we have the following series of categories:
\[\langle\vierpart\rangle \; \subset \;
\langle \fatcrosspart\rangle\ \; \subset \;
\langle \pi_l, l\leq 3 \rangle \; \subset \;
\langle \pi_l, l\leq 4 \rangle \; \subset \;
\ldots \; \subset \;
\langle \pi_l, l\in \NN \rangle \; \subset \;
\langle \primarypart \rangle\] 
In Section \ref{SectCStar} we will show that all these inclusions are strict. In particular, this proves that the categories $\langle\pi_l, l\leq k\rangle$ and $\langle \pi_l,l\in\N\rangle$ are non group-theoretical and hyperoctahedral. (Note that Section \ref{SectCStar} is independent from Section \ref{SectStructure} and \ref{SectWeakSimpl}.)
We will denote $\langle \pi_l, l\in \NN \rangle$ also by $\langle \pi_l, l\leq k \rangle$ with $k=\infty$ (formally, we might put $\pi_\infty:=\emptyset$).
We will now show that \emph{any} non group-theoretical hyperoctahedral category $\mathcal C$ is of the form $\langle\pi_l, l\leq k\rangle$ for some $k\in \{1,2,3,\ldots,\infty\}$, i.e.  the non group-theoretical hyperoctahedral categories of partitions are exactly given by the one-parameter series $\langle\pi_l, l\leq k\rangle$, $k\in\{1,2,\ldots,\infty\}$. 
Our strategy for the proof is the following. Suppose that $\mathcal C$ is a non group-theoretical hyperoctahedral category containing a partition $p$ which ``looks like a rough version of $\pi_k$''. We show that then $\pi_k\in\mathcal C$. Conversely, the category $\langle\pi_l, l\leq k\rangle$ generates all these distorted versions of $\pi_k$, and \emph{every} partition in $\mathcal C$ is such a distorted version. This proves that $\mathcal C=\langle\pi_l, l\leq k\rangle$ for some $k$.

The essence of the proof boils down to a detailed study of the structure of partitions in non group-theoretical hyperoctahedral categories. 
We begin with a simple observation about vertically reflected partitions.

\begin{definition}
 Let $p\in P(k,l)$ be a partition. Its \emph{vertically reflected version} $\bar p\in P(k,l)$ is given by vertical reflection, i.e. if the upper line of $p$ is of the form $a_{i(1)} a_{i(2)} \ldots a_{i(k)}$ and the lower line $a_{j(1)} a_{j(2)} \ldots a_{j(l)}$, then the upper line of $\bar p$ is $a_{i(k)} \ldots a_{i(2)} a_{i(1)}$ whereas the lower line is $a_{j(l)}\ldots a_{j(2)} a_{j(1)}$.
\end{definition}

\begin{lemma}\label{LemVertRefl}
 Let $\cC$ be a category of partitions and let $p\in P$ be a partition and $\bar p$ its vertical reflected version. Then $p\in \cC$ if and only if $\bar p\in \cC$.
\end{lemma}
\begin{proof}
If $p\in \cC$, then $p^*\in\cC$ (recall that $p^*$ is obtained by turning $p$ upside down) since $\cC$ is a category. Rotating $p^*$  yields $\bar p\in\cC$.
\end{proof}

The next two lemmata look rather technical, but in fact they encode the key observations for the classification of non group-theoretical hyperoctahedral categories: They are the tools to describe ``forbidden'' and ``allowed'' words of these categories.

\begin{lemma}\label{LemABA}
Let $\cC$ be a hyperoctahedral category of partitions and let $p\in P(0,n)$ be a partition of the form $p=abaX$, where $X$ is a subword and $a\neq b$. If $p\in\cC$, then $\primarypart\in\cC$.
More general, $\primarypart\in\cC$ whenever there is a partition $ab^kaX\in\cC$ for $k\in\NN$ odd.
\end{lemma}
\begin{proof}
By Lemma \ref{RemSizeTwo}, the block consisting of the points $b$ is not a singleton, hence $p$ is of the form $p=abaY_1bY_2$, where $Y_1$ and $Y_2$ are some subwords (which might contain $a$ or $b$). 
By Lemma \ref{LemVertRefl}, the vertical reflected version $\bar p$ of $p$ is in $\cC$, thus $p\otimes \bar p$ is in $\cC$ and it is of the form:
 \[p\otimes\bar p=abaY_1bY_2 \bar Y'_2 b' \bar Y'_1 a' b' a'\]
In $\bar p$, all letters $x$ of $p$ are replaced by letters $x'$ (since the blocks of $p$ and those of $\bar p$ are not connected to each other in $p\otimes\bar p$) and the order of the letters is reversed.

Using the pair partition, we may couple the last point of $Y_2$ with the first of $\bar Y'_2$. Iteratively, we obtain $abaa'ba'\in\cC$ (note that the blocks of $b$ and $b'$ are connecting by the coupling procedure). By Lemma \ref{LemBlockVerbinden}, we may connect the blocks of $a$ and $a'$ and we obtain $abaaba\in\cC$ which is a rotated version of $\primarypart$.

Using the pair partition, we deduce that $abaX\in\cC$ whenever $ab^kaX\in\cC$ for some odd number $k\in\NN$.
\end{proof}

\begin{lemma}[The Doubling Lemma]\label{LemABA2}
 Let $\cC$ be a non group-theoretical hyperoctahedral category of partitions. Let $p\in\cC$ be a partition of the form $p=X_1aX_2aX_3$ with some subwords $X_1, X_2$ and $X_3$. Then, the following holds true.
 \begin{itemize}
 \item[(a)] Every letter $b\neq a$ which appears in the subword $X_2$, appears an even number of times.
 \item[(b)] If $X_2$ is of odd length, it contains the letter $a$.
 \item[(c)] If $X_2$ is of even length (greater or equal two), it contains a consecutive pair of letters, i.e. $X_2$ is of the form $X_2=Y_1bbY_2$ (with possibly $b=a$).
 \end{itemize}
In particular, if a subword of a partition contains a letter twice, it contains a pair of consecutive letters.
\end{lemma}
\begin{proof}
 (a) We may assume that $X_2$ does not contain the letter $a$ (otherwise, we study $aX_2a=aY_1aY_2a\ldots aY_ma$ sectionwise). Let $b\neq a$ be a letter appearing in $X_2$ exactly  $k$ times. Hence, $p$ is of the following form:
 \[p=X_1aZ_1bZ_2b\ldots b Z_k b Z_{k+1} a X_3\]
 We apply the pair partition to the subwords $Z_i$ (note that $Z_i$ does not contain $a$ nor $b$), such that the following partition is in $\cC$, where the subwords $Z'_i$ have length either one or zero:
 \[p=X_1aZ'_1bZ'_2b\ldots b Z'_k b Z'_{k+1} a X_3\]
By Lemma \ref{LemABA}, the subwords $Z'_i$ must be empty for $i=2,\ldots, k$, thus $p$ is of the form:
 \[p=X_1aZ'_1b^k Z'_{k+1} a X_3\]
Using Lemma \ref{LemBlockVerbinden}, we connect the letters in $Z'_1$ and $Z'_{k+1}$ (if the subwords are non-empty) to the block of the letter $a$, and again by Lemma \ref{LemABA}, we infer that $k$ is even.

(b) Apply (a).

(c) If $X_2$ is of length two, it must be of the form $X_2=bb$ for some letter $b$, by (a). Otherwise, (a) implies that $X_2$ is of the form $X_2=Z_1bZ_2bZ_3$ for some letter $b$, where $Z_2$ does not contain $b$. Now, either $Z_2$ is empty, or it contains a pair by induction.
\end{proof}

By Lemma \ref{RemSizeTwo}, in hyperoctahedral categories all  partitions have even length. In non group-theoretical hyperoctahedral categories we have an even stronger statement: all blocks have even length.

\begin{lemma}\label{LemBlockLengthAllgem}
Let $\cC$ be a non group-theoretical hyperoctahedral category and let $p\in P(0,n)$ be a partition in $\cC$. Then each block of $p$ is of even length. 
\end{lemma}
\begin{proof}
Let $V$ be a block of $p$ labelled by the letter $a$. By rotation, $p$ is of the form:
\[p=aX_1aX_2a\ldots aX_{k-1}aX_k\]
The subwords $X_i$ do not contain $a$ and hence (by the Doubling Lemma \ref{LemABA2}) they are of even length. Using the pair partition, we infer that the partition $p'=a^kX_k'$ is in $\cC$, where $X_k'$ is some subword without the letter $a$. Since $\singleton\otimes\singleton\notin\cC$, we deduce that $k$ is even, hence $V$ is of even length.
\end{proof}

We deduce that all partitions $p$ in a non group-theoretical hyperoctahedral category $\cC$ are of a very special form. Suppose that $p\in P(0,n)$ is a word starting with the letter $a$. By Lemma \ref{LemBlockLengthAllgem},  the letter $a$ appears at least a second time in the word $p$, hence $p$ is of the form:
\[p=a X a Y\]
We may assume that $X$ does not contain the letter $a$, hence it is a subword starting with a letter $b$. By the Doubling Lemma \ref{LemABA2}, the letter $b$ appears in $X$ at least a second time, hence $p$ is of the form $p=abX'bX''aY$. Thus, a typical word in a non group-theoretical hyperoctahedral category is of the form:
\[p=abccddbaeffghhgeabba\]
We see that these partitions consist of pairs of letters which are nested into each others -- a structure which resembles the one of the noncrossing pair partitions $NC_2$ very much. The crucial difference is, that the pairs in our case can be connected to other pairs of the word (letters may appear more than only twice) -- even in a way involving crossings. Hence, the partitions in our focus can be thought of as being built in two steps: First, take a noncrossing pair partition on $n$ points, and then build blocks by combining some of these pairs. The question is, which linking of pairs is allowed.

As a tool for our imagination, let us recall the bijection between noncrossing pair partitions and Dyck paths. In our case, we only have a surjection from  partitions $p\in\cC$ (where $\cC$ is non group-theoretical and hyperoctahedral) to Dyck paths by decoding a partition $p=a_{i(1)}a_{i(2)}\ldots a_{i(n)}\in P(0,n)$ in the following way. We read the word letter by letter and we add an $(1,-1)$-move to the path at the $k$-th position, if the
$k$-th letter appears an odd number of times in the subword $a_{i(1)}\ldots a_{i(k)}$. Otherwise, we add a $(1,1)$-move. This way, we obtain a Dyck path of length $n$ (note that all blocks have even length). This assignment is not injective, but it illustrates how partitions in non group-theoretical categories look like. As an example, consider the above partition:
\[p=abccddbaeffghhgeabba\]
It yields the following Dyck path:

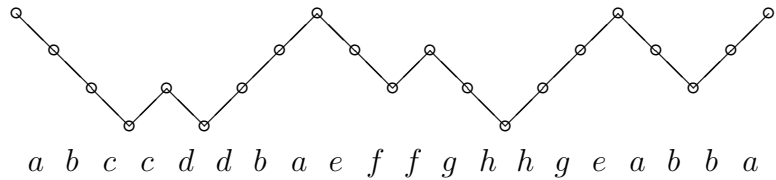
\setlength{\unitlength}{0.5cm}
\newsavebox{\Dyckpath}
\savebox{\Dyckpath}
{\begin{picture}(20,3)
    \put(0,3){\line(1,-1){1}} 
    \put(1,2){\line(1,-1){1}} 
    \put(2,1){\line(1,-1){1}} 
    \put(3,0){\line(1,1){1}}  
    \put(4,1){\line(1,-1){1}} 
    \put(5,0){\line(1,1){1}}  
    \put(6,1){\line(1,1){1}}  
    \put(7,2){\line(1,1){1}}  
    \put(8,3){\line(1,-1){1}} 
    \put(9,2){\line(1,-1){1}} 
    \put(10,1){\line(1,1){1}} 
    \put(11,2){\line(1,-1){1}}
    \put(12,1){\line(1,-1){1}}
    \put(13,0){\line(1,1){1}} 
    \put(14,1){\line(1,1){1}} 
    \put(15,2){\line(1,1){1}} 
    \put(16,3){\line(1,-1){1}}
    \put(17,2){\line(1,-1){1}}
    \put(18,1){\line(1,1){1}} 
    \put(19,2){\line(1,1){1}} 
\end{picture}}
\newsavebox{\Dyckpathpoints}
\savebox{\Dyckpathpoints}
{\begin{picture}(20,3)
    \put(0,3){$\circ$}
    \put(1,2){$\circ$}
    \put(2,1){$\circ$}
    \put(3,0){$\circ$}
    \put(4,1){$\circ$}
    \put(5,0){$\circ$}
    \put(6,1){$\circ$}
    \put(7,2){$\circ$}
    \put(8,3){$\circ$}
    \put(9,2){$\circ$}
    \put(10,1){$\circ$}
    \put(11,2){$\circ$}
    \put(12,1){$\circ$}
    \put(13,0){$\circ$}
    \put(14,1){$\circ$}
    \put(15,2){$\circ$}
    \put(16,3){$\circ$}
    \put(17,2){$\circ$}
    \put(18,1){$\circ$}
    \put(19,2){$\circ$}
    \put(20,3){$\circ$}
\end{picture}}

\begin{center}
  \begin{picture}(20,6)
   \put(0.2,1.2){\usebox{\Dyckpath}}
   \put(0,1){\usebox{\Dyckpathpoints}}
   \put(0.5,0){$a$}
   \put(1.5,0){$b$}
   \put(2.5,0){$c$}
   \put(3.5,0){$c$}
   \put(4.5,0){$d$}
   \put(5.5,0){$d$}
   \put(6.5,0){$b$}
   \put(7.5,0){$a$}
   \put(8.5,0){$e$}
   \put(9.5,0){$f$}
   \put(10.5,0){$f$}
   \put(11.5,0){$g$}
   \put(12.5,0){$h$}
   \put(13.5,0){$h$}
   \put(14.5,0){$g$}
   \put(15.5,0){$e$}
   \put(16.5,0){$a$}
   \put(17.5,0){$b$}
   \put(18.5,0){$b$}
   \put(19.5,0){$a$}
  \end{picture}
\end{center}

Note that the following partition yields the same path:
\[q=abccddbaaeebccbaijji\]

Recall that the partition $\pi_k\in P(0,4k)$ was defined by:
\[\pi_k:=a_1a_2\ldots a_ka_k\ldots a_2a_1a_1a_2\ldots a_ka_k\ldots a_2a_1\]
It gives rise to a very simple Dyck path which can be pictured as a ``W of depth $k$''. Here, $a_1a_2\ldots a_k$ describes the first of the four legs of the letter W. This motivates the next definition. Recall that we want to classify all non group-theoretical hyperoctahedral categories by ``distorted versions of $\pi_k$''. The philosophy is that a partition which contains a W of depth $k$ (in the following sense) also ``contains'' the structure of $\pi_k$ (to be made more precise in Lemma \ref{LemExistenzVonPi}). In the next definition, it is helpful to view the subwords $Y_i$ and $X_j^\alpha,X_j^\beta, X_j^\gamma,X_j^\delta$ as ``noise''. If they are all empty,  we exactly obtain the partition $\pi_k$.

\begin{definition}\label{Defwdepth}
 Let $p\in P(0,n)$ be a partition.
 We say that $p$ \emph{contains a W of depth} $k$, if there are $k$ different letters $a_1,\ldots a_k$ such that:
\[p=Y_1\underbrace{a_1X_1^\alpha a_2\ldots a_k}_{S_\alpha} X_k^\alpha \underbrace{a_k X_{k-1}^\beta a_{k-1}\ldots a_1}_{S_\beta} Y_2\underbrace{a_1X_1^\gamma a_2\ldots a_k}_{S_\gamma} X_k^\gamma \underbrace{a_k X_{k-1}^\delta a_{k-1}\ldots a_1}_{S_\delta} Y_3\]
More precisely: 
  \begin{itemize}
   \item[$\bullet$] $p$ is of the form  $p=Y_1S_\alpha X_k^\alpha S_\beta Y_2 S_\gamma X_k^\gamma S_\delta Y_3$,
   \item[$\bullet$] where $S_\alpha=a_1X_1^\alpha a_2X_2^\alpha \ldots a_k$, for some subwords $X_i^\alpha$,
   \item[$\bullet$] $S_\beta=a_k X_{k-1}^\beta a_{k-1}\ldots X_2^\beta a_2  X_1^\beta a_1$,
   \item[$\bullet$] $S_\gamma=a_1X_1^\gamma a_2X_2^\gamma \ldots a_k$,
   \item[$\bullet$] and $S_\delta=a_k X_{k-1}^\delta a_{k-1}\ldots X_2^\delta a_2 X_1^\delta a_1$.
   \item[$\bullet$] Furthermore, each letter $a_i$ for $i=1,\ldots k$ appears an \emph{odd number of times} in the subword $S_\alpha$, and likewise for $S_\beta, S_\gamma$ and $S_\delta$. (Note that  $a_i$ could appear in some subwords $X_i^\alpha,X_i^\beta,\ldots$)
  \end{itemize} 
  In this sense, $S_\alpha$, $S_\beta$, $S_\gamma$ and $S_\delta$ form the four legs of the (distorted) letter W. 
  We denote by $\wdepth(p)\in\NN_0$ the maximal $k$ such that a rotated version of $p$ contains a W of depth $k$. Recall that if $p$ is of the form
  \[p=b_{i(1)}b_{i(2)}\ldots b_{i(n)}\]
the partitions 
 \[b_{i(r)} b_{i(r+1)}\ldots b_{i(n)} b_{i(1)} b_{i(2)} \ldots b_{i(r-1)}\]
 are called rotated versions of $p$.
\end{definition}

\begin{remark} \label{RemPropertieswdepth}
We list some properties of $\wdepth$. 
 \begin{itemize}
 \item[(a)] We have $\wdepth(\pi_k)=k$ with $Y_i=X_j^\zeta=\emptyset$ for all $i,j$ and all $\zeta\in\{\alpha,\beta,\gamma,\delta\}$ in the above definition.
 \item[(b)] In the above example with
\[p=abccddbaeffghhgeabba,\quad\textnormal{and}\quad q=abccddbaaeebccbaijji\]
 we have $\wdepth(p)=2$ (it contains a $W$ on $a$ and $b$) and $\wdepth(q)=3$ (it contains a $W$ on $a, b$ and $c$). The numbers $\wdepth$ can not be read directly from the Dyck paths itself, we need a coloring of the paths in order to obtain them in a precise way. In this article, we use the Dyck paths only for the reason of illustrating the structures of partitions in non group-theoretical hyperoctahedral categories -- we do not use them for rigorous proofs.
 \item[(c)] We have $\wdepth(p)\leq \#\{\textnormal{blocks of }p\}$ for all $p\in P(0,n)$.
 \item[(d)] Let $p\in P(0,n)$ be a partition containing a block of length at least  four, then $\wdepth(p)\geq 1$.
 \item[(e)] If $p\in P(0,n)$ is noncrossing, then $\wdepth(p)\leq 1$. Conversely, if $\cC$ is a non group-theoretical hyperoctahedral category and $p\in\cC$ fulfills $\wdepth(p)\leq 1$, then $p$ is noncrossing (and hence $p\in\langle\vierpart\rangle$). Indeed, if $\wdepth(p)=0$, then  all blocks of $p$ are of length two by (d) and Lemma \ref{LemBlockLengthAllgem}. Hence, $p$ is a pair partition which is noncrossing due to the Doubling Lemma \ref{LemABA2}. For the case $\wdepth(p)=1$, assume that $p$ contains two crossing blocks. Hence $p$ is of the form (after rotation):
 \[p=aX_1bX_2aX_3bX_4\]
 By the Doubling Lemma \ref{LemABA2}, we infer that $X_1bX_2$ must contain the letter $b$ a second time. By rotation, the same is true for $X_3bX_4$ and likewise for the letter $a$ in the subwords in between two letters $b$. We finally deduce that $p$ is of the form:
 \[p=aY_1bY_2bY_3aY_4aY_5bY_6bY_7aY_8\]
 Here, we may assume that $Y_1$, $Y_3$, $Y_5$ and $Y_7$ do not contain the letters $a$ and $b$. Thus, $p$ contains a W of depth 2, contradicting $\wdepth(p)=1$.
 \item[(f)] The definition of $\wdepth$ is made in such a way that a partition $p$ with $\wdepth(p)=k$ is ``like $\pi_k$''. For this, the requirement on the letters to appear an odd number of times in the subwords $S_\alpha, S_\beta, S_\gamma$ and $S_\delta$ is crucial.  For instance, the partition 
\[a_1^2a_2^2\ldots a_k^2a_k^2\ldots a_2^2a_1^2a_1^2a_2^2\ldots a_k^2a_k^2\ldots a_2^2a_1^2\]
 is not like $\pi_k$, it is rather like a rotated version of $\pi_2$ (compare for instance the Dyck paths as a ``minimal'' requirement for a similarity). The point is, that it is not difficult to permute pairs of points in partitions in non group-theoretical hyperoctahedral categories (since they contain $\fatcrosspart$ whenever they are not $\langle\vierpart\rangle$). The important structure is the one of those nested pairs which are somehow ``trapped'' between single legs.
 \end{itemize}
\end{remark}

\section{Classification of non group-theoretical hyperoctahedral categories} \label{SectWeakSimpl}

In this section we classify non group-theoretical hyperoctahedral categories.
Let us give a motivation of one of the main techniques in the classification. Let $p$ be a partition containing a W of depth $k$. Then $p$ is of the form (after rotation):
\[p=(a_1 X_1^\alpha \ldots a_k) X_k^\alpha (a_k \ldots X_1^\beta a_1) Y_1
(a_1 X_1^\gamma \ldots a_k) X_k^\gamma (a_k \ldots X_1^\delta a_1) Y_2\]
Fix some $1\leq j\leq k-1$. We consider the following two subwords of $p$:
\[a_j X_j^\alpha a_{j+1} \ldots a_k X_k^\alpha a_k \ldots a_{j+1} X_j^\beta a_j\]
and
\[a_{j+1} \ldots a_k X_k^\alpha a_k \ldots a_{j+1}\]
Now, let us assume that the subword $X_j^\alpha$  contains a letter $x$ an odd number of times, and $x\neq a_j$, $x\neq a_{j+1}$. Since $x$ appears an even number of times in each of the  above two subwords, we infer that it must appear in $X_j^\beta$ an odd number of times, too. We proceed in the same way to show that it appears in  $X_j^\gamma$ and $X_j^\delta$ again an odd number of times. 
This argument will be used in the following proof in order to construct a  W of depth $k+1$ in $p$. The following scheme illustrates this ``odd and even''-technique.

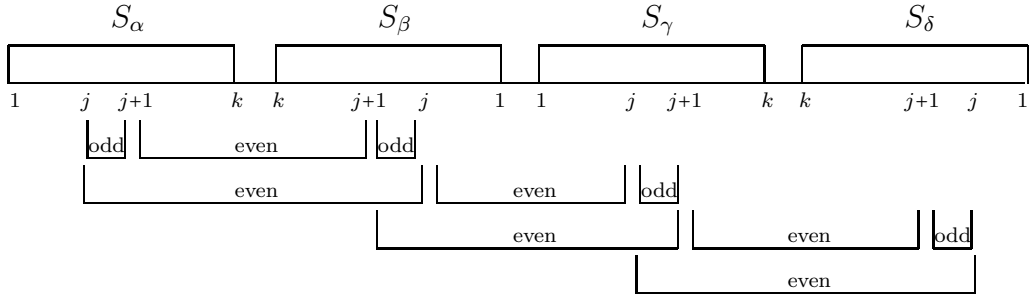
\setlength{\unitlength}{0.5cm}
\begin{center}
  \begin{picture}(27,8)
   \put(2.7,7){$S_\alpha$}
   \put(-0.3,5.5){\uppartii{1}{0}{6}}
   \put(9.8,7){$S_\beta$}
   \put(6.8,5.5){\uppartii{1}{0}{6}}
   \put(16.8,7){$S_\gamma$}
   \put(13.8,5.5){\uppartii{1}{0}{6}}
   \put(23.8,7){$S_\delta$}
   \put(20.8,5.5){\uppartii{1}{0}{6}}
   \put(0,5.5){\line(1,0){27}}
   \put(0,5){$_1$}
   \put(1.9,5){$_j$}
   \put(2.9,5){$_{j+1}$}
   \put(5.9,5){$_k$}
   \put(7,5){$_k$}
   \put(9.1,5){$_{j+1}$}
   \put(10.9,5){$_j$}
   \put(12.9,5){$_1$}
   \put(14,5){$_1$}
   \put(16.4,5){$_j$}
   \put(17.5,5){$_{j+1}$}
   \put(20,5){$_k$}
   \put(21,5){$_k$}
   \put(23.8,5){$_{j+1}$}
   \put(25.5,5){$_j$}
   \put(26.8,5){$_1$}
   \put(1.8,5.5){\partii{1}{0}{1}}
   \put(2.1,3.8){$_{\textnormal{odd}}$}
   \put(3.2,5.5){\partii{1}{0}{6}}
   \put(6,3.8){$_{\textnormal{even}}$}
   \put(9.5,5.5){\partii{1}{0}{1}}
   \put(9.8,3.8){$_{\textnormal{odd}}$}
   \put(1.7,4.3){\partii{1}{0}{9}}
   \put(6,2.6){$_{\textnormal{even}}$}
   \put(11.1,4.3){\partii{1}{0}{5}}
   \put(13.4,2.6){$_{\textnormal{even}}$}
   \put(16.5,4.3){\partii{1}{0}{1}}
   \put(16.8,2.6){$_{\textnormal{odd}}$}
   \put(9.5,3.1){\partii{1}{0}{8}}
   \put(13.4,1.4){$_{\textnormal{even}}$}
   \put(17.9,3.1){\partii{1}{0}{6}}
   \put(20.7,1.4){$_{\textnormal{even}}$}
   \put(24.3,3.1){\partii{1}{0}{1}}
   \put(24.6,1.4){$_{\textnormal{odd}}$}
   \put(16.4,1.9){\partii{1}{0}{9}}
   \put(20.7,0.2){$_{\textnormal{even}}$}
  \end{picture}
\end{center}

\begin{lemma}\label{LemExistenzVonPi}
Let $\cC$ be a non group-theoretical hyperoctahedral category. If $\cC$ contains a partition $p\in P(0,n)$  with $\wdepth(p)=k$, then $\pi_k\in\cC$.
\end{lemma}
\begin{proof}
For the case $k=1$, note that the partition $\pi_1=\vierpart$ is in all hyperoctahedral categories.
Let $k\geq 2$ and let $p\in\cC$ be the partition of minimal length $n$ containing a W of depth $k$. We prove that $p$ coincides with a rotated version of $\pi_k$. By definition, there are $k$ different letters $a_1,\ldots a_k$ such that $p$ is of the following form (after rotating):
\[p=S_\alpha X_k^\alpha S_\beta Y_1 S_\gamma X_k^\gamma S_\delta Y_2\]
Here, $S_\alpha=a_1X_1^\alpha a_2X_2^\alpha \ldots a_k$ and likewise $S_\beta, S_\gamma$ and $S_\delta$ are as in Definition \ref{Defwdepth}.
Each letter $a_i$ for $i=1,\ldots k$ appears an odd number of times in the subword $S_\alpha$, and likewise for $S_\beta, S_\gamma$ and $S_\delta$. 

\textbf{\emph{Step 1.} Direct consequences of the minimality of $p$.}

 All subwords $X_i^\zeta$ for all $\zeta\in\{\alpha,\beta,\gamma,\delta\}$ and all $i$, and also all $Y_j$ contain no two consecutive points belonging to the same block. Otherwise, it would contradict the minimality of $p$ since we could apply the pair partition to such a pair. Furthermore, there are no $a_j^3$ in $S_\zeta$. A priori, these subwords could appear for instance as $X_{j-1}^\alpha a_j X_j^\alpha={X'}_{j-1}^\alpha a_ja_j a_j{X'}_j^\alpha$, but we could then use the pair partition to erase $a_j^2$ (contradiction to the minimality of $p$). 
The subwords $Y_1$ and $Y_2$ are either empty or consist just of the single letter $a_1$ -- if $Y_1$ was of even length at least two, we would obtain a consecutive pair by the Doubling Lemma \ref{LemABA2}, since $S_\beta$ ends with $a_1$ and $S_\gamma$ begins with $a_1$. For $X_k^\alpha$ and $X_k^\gamma$, we deduce that they are either empty or consist just of the single letter $a_k$.

\textbf{\emph{Step 2.} No subword $S_\zeta$, $\zeta\in\{\alpha,\beta,\gamma,\delta\}$ contains a consecutive pair of letters.}

Assume the contrary.
This step is the most complicated in the proof. The idea is, that we can now find copies of one of the letters $a_i$ such that we can rearrange the W in $p$ -- and shorten $p$, contradicting the minimality of  $p$. 

Let us consider the case $\zeta=\alpha$. By step 1, we know that no subword $X_i^\alpha$ contains  a consecutive pair of letters. Thus, the pair in $S_\alpha$ needs to be formed of two copies of a letter $a_j$ one of which is either in  $X_{j-1}^\alpha$ or in $X_j^\alpha$. Hence, either $X_j^\alpha=a_j{X'}_j^\alpha$ or $X_{j-1}^\alpha={X'}_{j-1}^\alpha a_j$ for some $j$. Since $a_j$ appears an odd number of times in $S_\alpha$, it must appear at least a third time in $S_\alpha$. It cannot appear again in $X_{j-1}^\alpha$ or $X_j^\alpha$, since this would yield $a_j^3$ by the Doubling Lemma \ref{LemABA2}, in contradiction to step 1. We are thus in the situation that the letter $a_j$ appears at least three times in the subword
\[S_\alpha=a_1X_1^\alpha \ldots X_{j-2}^\alpha a_{j-1} X_{j-1}^\alpha a_j X_j^\alpha a_{j+1} X_{j+1}^\alpha \ldots a_k\]
and exactly twice in $X_{j-1}^\alpha a_j X_j^\alpha$. Hence it appears at least once either in the subword $a_1X_1^\alpha\ldots X_{j-2}^\alpha a_{j-1}$ or in $a_{j+1}X_{j+1}^\alpha\ldots a_k$.

First assume that $a_j$ appears in the subword $a_1X_1^\alpha  \ldots X_{j-2}^\alpha a_{j-1}$. Let $r$ be the smallest number such that $a_j$ appears in $X_r^\alpha$ (it can appear at most once in $X_r^\alpha$ by the Doubling Lemma \ref{LemABA2} since this subword contains no consecutive pairs by step 1). Since $a_j$ appears in $S_\alpha$ an odd number of times, it must appear an even number of times in the subword $a_{r+1}X_{r+1}^\alpha\ldots a_k$. On the other hand, it appears an even number of times in the subword 
\[a_{r}X_{r}^\alpha a_{r+1}X_{r+1}^\alpha\ldots  a_kX_k^\alpha a_k\ldots  X_{r+1}^\beta a_{r+1} X_r^\beta a_r,\]
as well as in the subword 
\[a_{r+1}X_{r+1}^\alpha\ldots  a_kX_k^\alpha a_k\ldots  X_{r+1}^\beta a_{r+1},\]
by the Doubling Lemma \ref{LemABA2}. Hence, it must appear exactly once in $X_r^\beta$ (see the explanation preceding the proof).

The same holds true for the cases $\gamma$ and $\delta$: Since $a_j$ appears an even number of times in
\[a_{r+1} X_{r}^\beta a_r X_{r-1}^\beta \ldots X_1^\beta a_1 Y_1  a_1 X_1^\gamma \ldots X_{r-1}^\gamma a_r X_r^\gamma a_{r+1},\]
as well as in 
\[ a_r X_{r-1}^\beta \ldots X_1^\beta a_1 Y_1  a_1 X_1^\gamma \ldots X_{r-1}^\gamma a_r,\]
it must appear in $X_r^\gamma$. In the same way, we deduce that $a_j$ appears in $X_r^\delta$. Now, using the pair partition, we erase the pair $a_j^2$ in the subword $X_{j-1}^\alpha a_j X_j^\alpha$ and we obtain a partition in $\cC$ containing a W of depth $k$ (where now the ``W-letters'' are given in the order $a_1,\ldots,a_r,a_j,a_{r+1},\ldots,a_{j-1},a_{j+1},\ldots,a_k$), contradicting the minimality of $p$.

We infer, that the third appearance of the letter $a_j$ in $S_\alpha$ is not in $a_1X_1^\alpha\ldots X_{j-2}^\alpha a_{j-1}$ but in the subword $a_{j+1}X_{j+1}^\alpha\ldots a_k$. Taking $s$ to be the maximal index such that $a_j$ appears in $X_s^\alpha$, we proceed in the same way as above to infer that $a_j$ appears in the other $X_s^\zeta$ for $\zeta\in\{\beta,\gamma,\delta\}$, too. This contradicts the minimality of $p$ and we conclude that no string $S_\zeta$ contains a consecutive pair of letters.

\textbf{\emph{Step 3.} No subword $X_i^\zeta$ contains a letter $a_j$.}

Assume that $a_j$ appears in $X_i^\zeta$ for some $i<k$ and some $\zeta$. Then, $S_\zeta$ contains a pair of consecutive points (using the Doubling Lemma \ref{LemABA2}), contradicting step 2. Furthermore, if $X_k^\alpha=a_k$, then $a_k$ appears an odd number of times in the string $a_{k-1}X_{k-1}^\alpha a_k X_k^\alpha a_k X_{k-1}^\beta a_{k-1}$, which contradicts the Doubling Lemma \ref{LemABA2}. Likewise we see that $X_k^\gamma$, $Y_1$ and $Y_2$ are empty.

\textbf{\emph{Step 4.} All $X_i^\zeta$ are empty.} 

Assume that $X_i^\alpha$ contains a letter $x\notin\{a_1,\ldots,a_k\}$. This letter appears only once in $S_\alpha$ since there would be a consecutive pair in $S_\alpha$ otherwise (by the Doubling Lemma \ref{LemABA2}). Proceeding like in step 2 (see also the motivation for this proof), we infer that $x$ must appear in the other $X_i^\zeta$ for $\zeta\in\{\beta,\gamma,\delta\}$, too. Hence, $p$ contains a W of depth $k+1$, which is a contradiction to $\wdepth(p)=k$.

We finally conclude that $p$ is of the form
\[p=S_\alpha X_k^\alpha S_\beta Y_1 S_\gamma X_k^\gamma S_\delta Y_2\;,\]
where 
\[S_\alpha=a_1X_1^\alpha a_2X_2^\alpha \ldots a_k=a_1a_2\ldots a_{k-1}a_k\;,\]
and similarly for $S_\beta, S_\gamma$ and $S_\delta$. Furthermore $X_k^\alpha=Y_1=X_k^\gamma=Y_2=\emptyset$ by step 3. Thus, $p=\pi_k$.
\end{proof}

For the converse direction, we use the following idea. Take a partition $p\in P(0,2n)$ in a non group-theoretical hyperoctahedral category. We want to prove that it is in  $\langle\pi_l,l\leq k\rangle$, whenever $\wdepth(p)\leq k$. The lower tips of the Dyck path associated to $p$ form pairs of consecutive letters, hence we are in situations $p=XaaY$. By induction with respect to the length of $p$ we may assume $q:=XY\in\langle\pi_l,l\leq k\rangle$.
If the block corresponding to the letter $a$ is of length two, we are done -- simply compose $q$ with $\idpart^{\otimes\alpha}\otimes\paarpart\otimes\idpart^{\otimes\beta}$ for suitable $\alpha$ and $\beta$. The resulting partition, namely $p$, is in $\langle\pi_l,l\leq k\rangle$.

Now, if the letter $a$ appears more than twice, $p$ is of the following form (after rotation):
\[p=aaXaYaZ\]
Again, we assume $p'=XaYaZ\in\langle\pi_l,l\leq k\rangle$ by induction hypothesis. Composing $p'$ with $\idpart^{\otimes\alpha}\otimes\paarpart\otimes\idpart^{\otimes\beta}$ and connecting neighbouring blocks (see Lemma \ref{LemBlockVerbinden}) yields that the following partition is in $\langle\pi_l,l\leq k\rangle$:
\[q=XaaaYaZ\]
We are left with the problem of shifting the pair $aa$ from the right hand side of $X$ to the left hand side of it in order to obtain $p\in\langle\pi_l,l\leq k\rangle$. For doing so, note that the following partition -- a rotated version of $\pi_k\in P(0,4k)$ -- enables us to let the words $a_2\ldots a_ka_k\ldots a_2$ and $a_1a_1$ commute (via composition):
\begin{align*}
\begin{pmatrix} (a_1 a_1) & (a_2 \ldots a_k a_k \ldots a_2)\\
(a_2 \ldots a_k a_k \ldots a_2) &(a_1 a_1) \end{pmatrix} \in P(2k,2k)
\end{align*}
The problem is, that $X$ is not a product of ``proper V's'', i.e. of words of the form $a_2\ldots a_ka_k\ldots a_2$. Nevertheless, analyzing the Dyck path structure of $X$, we can view it as a product $X_1X_2\ldots X_l$ of ``skew V's''.
As an example, consider:
\[X=bccdeeffdb\]
Its Dyck path is of the form:

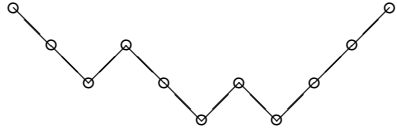
\setlength{\unitlength}{0.5cm}
\newsavebox{\DyckpathZwei}
\savebox{\DyckpathZwei}
{\begin{picture}(10,3)
    \put(0,3){\line(1,-1){1}} 
    \put(1,2){\line(1,-1){1}} 
    \put(2,1){\line(1,1){1}} 
    \put(3,2){\line(1,-1){1}}  
    \put(4,1){\line(1,-1){1}} 
    \put(5,0){\line(1,1){1}}  
    \put(6,1){\line(1,-1){1}}  
    \put(7,0){\line(1,1){1}}  
    \put(8,1){\line(1,1){1}} 
    \put(9,2){\line(1,1){1}} 
\end{picture}}
\newsavebox{\DyckpathpointsZwei}
\savebox{\DyckpathpointsZwei}
{\begin{picture}(10,3)
    \put(0,3){$\circ$}
    \put(1,2){$\circ$}
    \put(2,1){$\circ$}
    \put(3,2){$\circ$}
    \put(4,1){$\circ$}
    \put(5,0){$\circ$}
    \put(6,1){$\circ$}
    \put(7,0){$\circ$}
    \put(8,1){$\circ$}
    \put(9,2){$\circ$}   
    \put(10,3){$\circ$}
\end{picture}}

\begin{center}
  \begin{picture}(10,5)
   \put(0.2,1.2){\usebox{\DyckpathZwei}}
   \put(0,1){\usebox{\DyckpathpointsZwei}}
  \end{picture}
\end{center}
We can view $X$ as a product of $X_1=bcc$, $X_2=dee$ and $X_3=ffdb$, hence decomposing the Dyck path into:
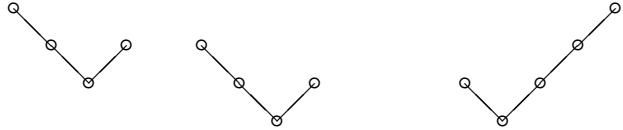
\setlength{\unitlength}{0.5cm}
\newsavebox{\DyckpathZweib}
\savebox{\DyckpathZweib}
{\begin{picture}(16,3)
    \put(0,3){\line(1,-1){1}} 
    \put(1,2){\line(1,-1){1}} 
    \put(2,1){\line(1,1){1}} 
    \put(5,2){\line(1,-1){1}}  
    \put(6,1){\line(1,-1){1}} 
    \put(7,0){\line(1,1){1}}  
    \put(12,1){\line(1,-1){1}}  
    \put(13,0){\line(1,1){1}}  
    \put(14,1){\line(1,1){1}} 
    \put(15,2){\line(1,1){1}} 
\end{picture}}
\newsavebox{\DyckpathpointsZweib}
\savebox{\DyckpathpointsZweib}
{\begin{picture}(12,3)
    \put(0,3){$\circ$}
    \put(1,2){$\circ$}
    \put(2,1){$\circ$}
    \put(3,2){$\circ$}
    \put(5,2){$\circ$}
    \put(6,1){$\circ$}
    \put(7,0){$\circ$}
    \put(8,1){$\circ$}
    \put(12,1){$\circ$}
    \put(13,0){$\circ$}
    \put(14,1){$\circ$}
    \put(15,2){$\circ$}   
    \put(16,3){$\circ$}
\end{picture}}

\begin{center}
  \begin{picture}(16,5)
   \put(0.2,1.2){\usebox{\DyckpathZweib}}
   \put(0,1){\usebox{\DyckpathpointsZweib}}
  \end{picture}
\end{center}
If we now insert suitable pairs iteratively in $X$ (at some point we have to prove that we can do so), we obtain a product $X'_1\ldots X'_l$ of honest V's such that each $X'_i$ commutes with $aa$.
In our example this would yield:
\[X'=(bccb)(bdeedb)(bdffdb)\]
With associated Dyck path:

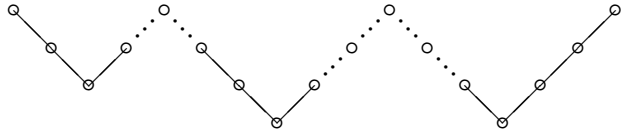
\setlength{\unitlength}{0.5cm}
\newsavebox{\DyckpathDrei}
\savebox{\DyckpathDrei}
{\begin{picture}(16,3)
    \put(0,3){\line(1,-1){1}} 
    \put(1,2){\line(1,-1){1}} 
    \put(2,1){\line(1,1){1}} 
    \put(3.2,2.1){$\cdot$} 
    \put(3.4,2.3){$\cdot$}
    \put(3.6,2.5){$\cdot$} 
    \put(4.2,2.5){$\cdot$} 
    \put(4.4,2.3){$\cdot$}
    \put(4.6,2.1){$\cdot$} 
    \put(5,2){\line(1,-1){1}}  
    \put(6,1){\line(1,-1){1}} 
    \put(7,0){\line(1,1){1}}  
    \put(8.2,1.1){$\cdot$} 
    \put(8.4,1.3){$\cdot$}
    \put(8.6,1.5){$\cdot$} 
    \put(9.2,2.1){$\cdot$} 
    \put(9.4,2.3){$\cdot$}
    \put(9.6,2.5){$\cdot$} 
    \put(10.2,2.5){$\cdot$} 
    \put(10.4,2.3){$\cdot$}
    \put(10.6,2.1){$\cdot$} 
    \put(11.2,1.5){$\cdot$} 
    \put(11.4,1.3){$\cdot$}
    \put(11.6,1.1){$\cdot$} 
    \put(12,1){\line(1,-1){1}}  
    \put(13,0){\line(1,1){1}}  
    \put(14,1){\line(1,1){1}} 
    \put(15,2){\line(1,1){1}} 
\end{picture}}
\newsavebox{\DyckpathpointsDrei}
\savebox{\DyckpathpointsDrei}
{\begin{picture}(10,3)
    \put(0,3){$\circ$}
    \put(1,2){$\circ$}
    \put(2,1){$\circ$}
    \put(3,2){$\circ$}
    \put(4,3){$\circ$}
    \put(5,2){$\circ$}
    \put(6,1){$\circ$}
    \put(7,0){$\circ$}
    \put(8,1){$\circ$}
    \put(9,2){$\circ$}
    \put(10,3){$\circ$}
    \put(11,2){$\circ$}
    \put(12,1){$\circ$}
    \put(13,0){$\circ$}
    \put(14,1){$\circ$}
    \put(15,2){$\circ$}
    \put(16,3){$\circ$}
\end{picture}}

\begin{center}
  \begin{picture}(16,5)
   \put(0.2,1.2){\usebox{\DyckpathDrei}}
   \put(0,1){\usebox{\DyckpathpointsDrei}}
  \end{picture}
\end{center}

Recall the definition of partitions in single-double leg form (Definition \ref{DefSingleDouble}).

\begin{lemma}\label{LemWdepthInPi}
 Let $\cC$ be a non group-theoretical hyperoctahedral category of partitions and let $k\geq 1$. Let $p\in P(0,2n)\cap\cC$ be a  partition in single-double leg form with $\wdepth(p)\leq k$. Then $p\in\langle\pi_l,l\leq k\rangle$.
\end{lemma}
\begin{proof}
We begin the proof by analyzing some special cases.
If $\wdepth(p)\leq 1$, then $p\in\langle\vierpart\rangle=\langle\pi_1\rangle\subset\langle\pi_l,l\leq k\rangle$ by Remark \ref{RemPropertieswdepth}(e). Furthermore, if $p$ contains only one block, then $\wdepth(p)\leq 1$. Hence, we can restrict to partitions $p$ which contain at least two blocks. If there is at most one block which has length greater or equal to four, then $\wdepth(p)\leq 1$. Therefore, we can restrict to the cases where $\wdepth(p)\geq 2$, $k\geq 2$ and $p$ containing at least two blocks of length greater or equal to four -- thus the length of $p$ is at least eight.
 
\textbf{\emph{Step 1.} Decomposition of partitions in $\CC$.}

We prove the lemma by induction on the length $2n$ of $p$. If $n=4$, then $p$ contains exactly two blocks of length four, and $\wdepth(p)=2$. Hence $p=\pi_2$ up to rotation, thus $p\in\langle\pi_l,l\leq k\rangle$. 

In the induction step, we assume that $p$ contains no pairs of consecutive letters appearing only as a block of size two -- hence we exclude the situation $p=XaaY$, where the subwords $X$ and $Y$ do not contain the letter $a$.
Indeed, in this case, we would have $\wdepth(XY)\leq\wdepth(p)$, and thus $q:=XY\in\langle\pi_l,l\leq k\rangle$ by the induction hypothesis. Composing $q$ with a suitable partition $\idpart^{\otimes\alpha}\otimes\paarpart\otimes\idpart^{\otimes\beta}$ yields $p\in\langle\pi_l,l\leq k\rangle$.

Let us now begin with the main part of the proof. We are given a partition $p\in\cC$ of the following form (after rotation and possible reflection, see Lemma \ref{LemVertRefl}):
\[p=aaXaYaZ\]
We may assume that $X$ and $Y$ do not contain the letter $a$, and that $X$ contains at least one letter different from $a$, since $p$ is in single-double leg form. Furthermore, we may assume that  if $bb$ is a pair of consecutive letters in $X=X_1bbX_2$, then the letter $b$ appears at least once in $Z$. Otherwise, we study the block on the letter $b$ instead of the one on the letter $a$. This is a minimality assumption on the length of $X$. Note that $b$ cannot appear in $Y$, since this would contradict the Doubling Lemma \ref{LemABA2} -- the letter $a$ appears in $XaY$ exactly once.

By the induction hypothesis, we know that the partition $XaYaZ$ is in $\langle\pi_l,l\leq k\rangle$, and hence the following partition, too (see the above exposition of the idea of the proof or Lemma \ref{LemReduzieren1}):
\[q:=XaaaYaZ\] 
We have to prove that also $p$ is in $\langle\pi_l,l\leq k\rangle$. To see this, we form subwords $X_j$ of $X$, such that $X$ is of the form:
\[X=X_1 X_2 \ldots X_l\]
The subwords $X_j$ are obtained by the following procedure: We read $X$ from right to left, writing down letter by letter, until one of the letter appears for the second time. It can only be the last written down, by the Doubling Lemma \ref{LemABA2}. For instance for the first subword from the right, $X_l$, we obtain:
\[b_mb_m \ldots  b_2b_1\]
We continue writing down the letters (which now come in reversed order) as long as they appear for the second time in our subword. As soon as a letter different from all the others appears, or a letter appears for the third time, we end our subword and this will be the starting point of our next subword.
In the end, the subword $X_l$ will be of the form:
 \[X_l=b_s \ldots b_mb_m \ldots b_s b_{s-1} \ldots b_2b_1\]
Here, $1\leq s\leq m$.
 The subword $X_{l-1}$ then either starts with $b_s$ or with a letter different from all $b_1,\ldots,b_m$.
 In general, the subwords are of the form:
\[X_j=b^{(j)}_{s_j} b^{(j)}_{s_j+1} \ldots b^{(j)}_{m_j}b^{(j)}_{m_j} \ldots b^{(j)}_{t_j+1}b^{(j)}_{t_j}\]
Here, $1\leq s_j, t_j\leq m_j$ with $s_j\leq t_j$ or $s_j\geq t_j$, and all letters $b^{(j)}_r$ are mutually different.
Viewing $X$ as a Dyck path, each subword $X_j$ corresponds exactly to one ``tooth'' of the path, as motivated above. Hence, we could also view the construction of the subwords $X_j$ as reading in an ``even/up'' and in an ``odd/down'' mode, depending on whether we go $(1,-1)$ or $(1,1)$ in the Dyck path. We always start with a new subword whenever we switch from an even reading mode into an odd reading mode.
 
Considering the subword $X_l$ we now prove that $m\leq k-1$.

\textbf{\emph{Step 2.} We have $m\leq k-1$ for the subword $X_l=b_s \ldots b_mb_m \ldots b_1$.}

In the sequel, we will construct a W of depth $m+1$ in $p$. 
We are in the situation that $p$ is of the form:
\[p=aaX_1\ldots X_{l-1} b_s \ldots b_mb_m \ldots b_s \ldots b_2b_1aYaZ\]
By the minimality assumption on $p$ regarding the pair $aa$, we know that $b_m$ appears in the subword $Z$. Recall that it cannot appear in $Y$ by the Doubling Lemma \ref{LemABA2}. Hence $Z$ is of the form $Z=Z'b_mZ_m$. We may assume that $b_m$ does not appear in $Z'$. The letter $b_{m-1}$ appears in the subword
\[b_mb_{m-1}\ldots b_2b_1aYaZ'b_m\]
an even number of times, likewise in $aYa$. Thus, it appears in $Z'$ at least once -- or to be more precise: an odd number of times. Inductively we obtain that $p$ is of the form:
\[p=aaX_1\ldots X_{l-1} b_s \ldots b_mb_m \ldots b_2b_1aYaZ_0b_1Z_1b_2Z_2\ldots b_{m-1} Z_{m-1} b_m Z_m\]
By construction, we can obtain that $b_j$ does not appear in $Z_k$ for $k<j$.
Hence, every letter $b_j$ appears an odd number of times in the subword:
\[b_1Z_1b_2Z_2\ldots b_{m-1} Z_{m-1} b_{m}\]
Using rotation, we can consider the subword:
\[aZ_0b_1Z_1b_2Z_2\ldots b_{m-1} Z_{m-1} b_m Z_ma\]
By the Doubling Lemma \ref{LemABA2}, we know that $b_1$ appears in the above subword at least a second time. In each of the subwords $Z_k$ for $k<m$, the letter $b_1$ appears an even number of times (or not at all). This simply follows from the Doubling Lemma \ref{LemABA2} when considering the subwords
\[b_k\ldots b_1 a Y a Z_0b_1Z_1b_2Z_2\ldots b_{k-1}Z_{k-1}b_k\]
and
\[b_{k+1}b_k\ldots b_1 a Y a Z_0b_1Z_1b_2Z_2\ldots b_{k-1}Z_{k-1}b_kZ_kb_{k+1}\]
We infer that $b_1$ appears in the subword $Z_m$ an odd number of times, hence at least once.
Therefore, $Z_m=Z_m'b_1\tilde Z_0$. Inductively, we obtain:
\[Z_m=\tilde Z_m b_m \tilde Z_{m-1} b_{m-1}\ldots \tilde Z_2 b_2 \tilde Z_1 b_1 \tilde Z_0\]
Since the letters $b_1,\ldots b_{s-1}$ appear in $X_1\ldots X_{l-1}$, we infer that a rotated version of $p$ contains a W of depth $m+1$ consisting in the letters:
\[a, b_1, b_2,\ldots, b_m\]
Thus, $m+1\leq \wdepth(p)\leq k$. 

\textbf{\emph{Step 3.} Commuting $aa$ with the subword $X_l$.}

Let us go back to the following picture of $q\in\langle\pi_l,l\leq k\rangle$ -- the parentheses are just for readability:
\[q=X_1\ldots X_{l-1} \left(b_s \ldots b_mb_m \ldots b_s\right) b_{s-1}\ldots b_1aaaYaZ\]
Composing $q$ with suitable  $\idpart^{\otimes\alpha}\otimes\paarpart\otimes\idpart^{\otimes\beta}$ and using Lemma \ref{LemBlockVerbinden}, we may insert the pair $b_{s-1}b_{s-1}$ and deduce that the following partition is in $\langle\pi_l,l\leq k\rangle$, too:
\[X_1\ldots X_{l-1} \left(b_s \ldots b_mb_m \ldots b_s\right) \left(b_{s-1}b_{s-1}\right)b_{s-1}\ldots b_1aaaYaZ\]
Since $t:=m-s+1\leq m\leq k-1$, we use a rotated version of $\pi_{t+1}\in\langle\pi_l,l\leq k\rangle$ and compose it with the above partition in order to obtain the following partition in $\langle\pi_l,l\leq k\rangle$ (see the above exposition of the idea of the proof):
\[X_1\ldots X_{l-1} \left(b_{s-1}b_{s-1}\right)\left( b_s \ldots b_mb_m \ldots b_s\right) b_{s-1}\ldots b_1aaaYaZ\]
Iterating this procedure, we obtain the following partition in $\langle\pi_l,l\leq k\rangle$:
\[X_1\ldots X_{l-1} b_{s-1}\ldots b_1 b_1 \ldots b_{s-1} b_s \ldots b_mb_m \ldots b_s b_{s-1}\ldots b_1aaaYaZ\]
As $m\leq k-1$, we infer that we can now let $aa$ and $b_1\ldots b_m b_m \ldots b_1$ commute such that the following partition is in $\langle\pi_l,l\leq k\rangle$:
\[X_1\ldots X_{l-1} b_{s-1}\ldots b_1 aa b_1 \ldots b_{s-1} b_s \ldots b_mb_m \ldots b_s b_{s-1}\ldots b_1aYaZ\]
We can also write it like:
\[X_1\ldots X_{l-1} b_{s-1}\ldots b_1 aa b_1 \ldots b_{s-1} X_laYaZ\]

\textbf{\emph{Step 4.} Commuting $aa$ with the subword $X$.}

The scheme of Step 2 + 3 is the following. Starting with
\[X_1\ldots X_{l-2}X_{l-1} X_l aaa Y a Z\in\langle\pi_l,l\leq k\rangle\;,\]
we infer that
\[X_1\ldots X_{l-2}X_{l-1} b_{s-1}\ldots b_1 aab_1\ldots b_{s-1} X_la Y a Z\in\langle\pi_l,l\leq k\rangle\;.\]
Now, renaming $X_{l-1}':=X_{l-1}b_{s-1}\ldots b_1$ and applying Step 2 and 3 again yields
\[X_1\ldots X_{l-2} b'_{s'-1}\ldots b'_1 aab'_1\ldots b'_{s'-1} X_{l-1}'b_1\ldots b_{s-1}X_la Y aZ \in\langle\pi_l,l\leq k\rangle\;,\]
which amounts to:
\[X_1\ldots X_{l-2} b'_{s'-1}\ldots b'_1 aab'_1\ldots b'_{s'-1} X_{l-1}b_{s-1}\ldots b_1b_1\ldots b_{s-1}X_la Y aZ \in\langle\pi_l,l\leq k\rangle\]
Using the pair partition, we infer: 
\[X_1\ldots X_{l-2} b'_{s'-1}\ldots b'_1 aab'_1\ldots b'_{s'-1} X_{l-1}X_la Y a Z\in\langle\pi_l,l\leq k\rangle\]
An induction on steps 2 and 3 completes the proof that 
\[p= aa X_1\ldots X_{l-1} X_laYaZ\]
is in $\langle\pi_l,l\leq k\rangle$.
\end{proof}

Using Lemma \ref{LemExistenzVonPi} and Lemma \ref{LemWdepthInPi}, we can now classify all non group-theoretical hyperoctahedral categories.

\begin{theorem}\label{ThMain}
 Let $\cC$ be a non group-theoretical hyperoctahedral category. Then $\CC=\langle\pi_l,l\leq k\rangle$ for some $k\in\{1,2,\ldots,\infty\}$.
\end{theorem}
\begin{proof}
Let $k$ be the supremum of all numbers $\wdepth(p)$, where $p\in\cC$. Then $k\geq 1$, since $\wdepth(\vierpart)=1$. Let $p\in\cC$ and let $p'\in P(0,2n)$ be a rotated version of it. Let $p''$ be its associated partition in single-double leg form. By Lemma \ref{LemReduzieren2}, we have $p''\in\CC$ and hence $\wdepth(p'')\leq k$. By Lemma \ref{LemWdepthInPi}, we have $p''\in\langle\pi_l,l\leq k\rangle$. This implies $p\in\langle\pi_l,l\leq k\rangle$. We infer $\cC\subset \langle \pi_l, l\leq k\rangle$.

On the other hand, if $k<\infty$, then there is a partition $p\in\cC$ such that $\wdepth(p)=k$. By Lemma \ref{LemExistenzVonPi}, we infer that $\pi_k\in\cC$, which proves $\cC= \langle\pi_l,l\leq k\rangle$. If $k=\infty$, we find partitions $p_n\in\cC$ with $\wdepth(p_n)=l_n$ for $l_n\to\infty$, if $n\to\infty$. This yields  $\cC=\langle \pi_l,l\in\N\rangle$ in the case $k=\infty$.
\end{proof}

As a corollary, we obtain the main theorem of our article.

\begin{theorem}\label{ThMain2}
If $G$ is an orthogonal easy quantum group, its corresponding category of partitions
\begin{itemize}
\item[(i)] either is non-hyperoctahedral (and hence it is one of the 13 cases of \cite{Weber13}),
\item[(ii)] or it coincides with $\langle\pi_l,l\leq k\rangle$ for some $k\in\{1,2,\ldots,\infty\}$,
\item[(iii)] or it is group-theoretical and hyperoctahedral (see \cite{RaumWeberSimpl}, \cite{RaumWeberSemiDirect}).
\end{itemize}
\end{theorem}

Coming back to the description of the elements in $\langle \pi_l, l\in\N\rangle$, we now understand how to obtain them starting with noncrossing pair partitions. Take a noncrossing pair partition $p\in NC_2$ and consider its Dyck path. We now may connect pairs (such that they belong to the same block) according to the following rule:

\begin{itemize}
 \item[(1)] We may only connect pairs which are on the same level of the Dyck path.
 \item[(2)] If we connect two pairs, we must also connect all ``open pairs'' in between.
\end{itemize}

To illustrate this rule, let us consider the following partition $q_0\in NC_2$.
\[q_0=abccddbaeffghhgeijji\]

Its Dyck path is of the form:

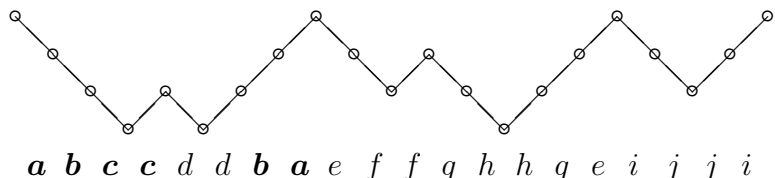
\setlength{\unitlength}{0.5cm}
\begin{center}
  \begin{picture}(20,5)
   \put(0.2,1.2){\usebox{\Dyckpath}}
   \put(0,1){\usebox{\Dyckpathpoints}}
   \put(0.5,0){$\boldsymbol a$}
   \put(1.5,0){$\boldsymbol b$}
   \put(2.5,0){$\boldsymbol c$}
   \put(3.5,0){$\boldsymbol c$}
   \put(4.5,0){$d$}
   \put(5.5,0){$d$}
   \put(6.5,0){$\boldsymbol b$}
   \put(7.5,0){$\boldsymbol a$}
   \put(8.5,0){$\underline e$}
   \put(9.5,0){$f$}
   \put(10.5,0){$f$}
   \put(11.5,0){$\underline g$}
   \put(12.5,0){$\underline h$}
   \put(13.5,0){$\underline h$}
   \put(14.5,0){$\underline g$}
   \put(15.5,0){$\underline e$}
   \put(16.5,0){$i$}
   \put(17.5,0){$j$}
   \put(18.5,0){$j$}
   \put(19.5,0){$i$}
  \end{picture}
\end{center}

We may now connect the pair $cc$ to the pair $hh$ since they are on the same level -- for instance, $cc$ with $gg$ would be impossible, as well as $cc$ with $jj$ (but $gg$ and $jj$ would be okay). But according to item (2), we also have to connect $bb$ with $gg$ as well as $aa$ with $ee$. We end up with the partition $q\in\langle\pi_l,l\in\N\rangle$ of the form:
\[q=abccddbaaffbccbaijji\]

Partitions in $\langle\pi_l,l\leq k\rangle$ may be obtained in the same way, but here the level in item (1) and the number of ``open pairs'' in item (2) are restricted.

Another way of viewing partitions in $\langle\pi_l,l\leq k\rangle$ is via rooted trees. ``Glueing'' together the pairs in the above Dyck path yields the following rooted graph:

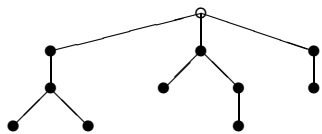
\setlength{\unitlength}{0.5cm}
\begin{center}
  \begin{picture}(8,4)
   \put(0,0){$\bullet$}
   \put(1,1){$\bullet$}
   \put(1,2){$\bullet$}
   \put(2,0){$\bullet$}
   \put(4,1){$\bullet$}
   \put(5,2){$\bullet$}
   \put(5,3){$\circ$}
   \put(6,0){$\bullet$}
   \put(6,1){$\bullet$}
   \put(8,1){$\bullet$}
   \put(8,2){$\bullet$}
   \put(0.2,0.2){\line(1,1){1}}
   \put(1.2,1.2){\line(1,-1){1}}
   \put(1.2,1.2){\line(0,1){1}}
   \put(1.2,2.2){\line(4,1){4}}
   \put(4.2,1.2){\line(1,1){1}}
   \put(5.2,2.2){\line(0,1){1}}
   \put(5.2,2.2){\line(1,-1){1}}
   \put(6.2,0.2){\line(0,1){1}}
   \put(5.2,3.2){\line(3,-1){3}}
   \put(8.2,1.2){\line(0,1){1}}
  \end{picture}
\end{center}
We read such a graph in the following way.
Starting from the root ($\circ$), we take the leftmost branch and we begin to ``open pairs'' until we reach the end of the most left branch. In our example: $q_0=abc$. We then return and go back to the first knot where a new branch begins: $q_0=abcc$. We walk along this new branch until its end ($q_0=abccd$) and return, if we reach its end, again going back to the first knot where a new branch begins ($q_0=abccddba$). Here, ``going back closes the pairs''. Step by step, we read the full word $q_0$ as above. In the tree associated to $q$ in turn, some edges are marked with the same color (in the picture the ``colors'' are $\alpha, \beta$ and $\gamma$) corresponding to related pairs.

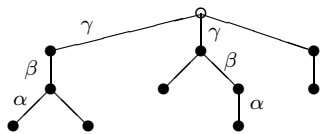
\setlength{\unitlength}{0.5cm}
\begin{center}
  \begin{picture}(8,4)
   \put(0,0){$\bullet$}
   \put(1,1){$\bullet$}
   \put(1,2){$\bullet$}
   \put(2,0){$\bullet$}
   \put(4,1){$\bullet$}
   \put(5,2){$\bullet$}
   \put(5,3){$\circ$}
   \put(6,0){$\bullet$}
   \put(6,1){$\bullet$}
   \put(8,1){$\bullet$}
   \put(8,2){$\bullet$}
   \put(0.2,0.2){\line(1,1){1}}
   \put(1.2,1.2){\line(1,-1){1}}
   \put(1.2,1.2){\line(0,1){1}}
   \put(1.2,2.2){\line(4,1){4}}
   \put(4.2,1.2){\line(1,1){1}}
   \put(5.2,2.2){\line(0,1){1}}
   \put(5.2,2.2){\line(1,-1){1}}
   \put(6.2,0.2){\line(0,1){1}}
   \put(5.2,3.2){\line(3,-1){3}}
   \put(8.2,1.2){\line(0,1){1}}
   \put(0.2,0.9){$_\alpha$}
   \put(6.5,0.8){$_\alpha$}
   \put(0.5,1.7){$_\beta$}
   \put(5.8,1.9){$_\beta$}
   \put(2,2.9){$_\gamma$}
   \put(5.4,2.7){$_\gamma$}
  \end{picture}
\end{center}

The coloring rule of such a rooted tree is obviously that if two edges have the same color, the paths that connects them with the root have to be colored in the same way. Furthermore, no two edges of a path from a vertex to the root are allowed to have the same color. Now, it is easy to see when a colored, rooted tree contains a $W$ of depth $k$: It is the maximal length of a  coloring of a path which appears at least twice. The only difficulty when working with the identification of colored, rooted trees and partitions in $\langle\pi_l,l\leq k\rangle$ is the correct change of colors when considering rotations of the partitions (note that this is needed to determine the value $\wdepth(p)$).

Let us remark a second thing. We say that a category $\CC$ is \emph{finitely generated}, if there are finitely many partitions $p_1,\ldots,p_n$ such that $\CC=\langle p_1,\ldots,p_n\rangle$, and $\CC$ is \emph{singly generated}, if $\CC=\langle p\rangle$ for some partition $p$. We have a dichotomy between singly generated and not finitely generated categories.

\begin{lemma}
 Let $\CC$ be a category of partitions. Then, $\CC$ is finitely generated if and only if it is singly generated. Hence $\CC$ is either singly generated or not finitely generated.
\end{lemma}
\begin{proof}
 Let $\CC=\langle p_1,\ldots, p_n\rangle$ and put $p:=p_1\otimes \ldots \otimes p_n$. Then $p\in\CC$. If all partitions $p_i$ are of even length, then $p_i\in\langle p \rangle$ using the pair partition. Hence $\CC=\langle p \rangle$. Otherwise, if one of the partitions $p_i$ is of odd length, then $\singleton\in\CC$. Putting $q:=p$ if $p$ is of odd length and $q:=p\otimes \singleton$ otherwise, we infer $\CC=\langle q\rangle$.
\end{proof}

In the next section, we will prove that $\langle\pi_l,l\leq k\rangle \neq\langle \pi_l, l\in\N\rangle$ for $k<\infty$. Using this, we see that $\langle \pi_l,l\in\N\rangle$ is not  finitely generated. In fact, this is the first concrete example of a category which is not singly generated, although there must be uncountably many, due to \cite{RaumWeberSemiDirect}. 

\begin{proposition}
 The category $\langle\pi_l,l\in\N\rangle$ is not finitely generated.
\end{proposition}
\begin{proof}
 Assume that $\langle\pi_l,l\in\N\rangle=\langle p \rangle$ for some partition $p\in P(0,2n)$ in single-double leg form. Let $k:=\wdepth(p)$. By Lemma \ref{LemWdepthInPi}, we know that $p\in\langle\pi_l,l\leq k\rangle$. Conversely, Lemma \ref{LemExistenzVonPi} shows that $\pi_k\in\langle p\rangle$. Thus $\langle\pi_l,l\in\N\rangle=\langle\pi_l,l\leq k\rangle$ which is a contradiction to Proposition \ref{PiKDifferent1}.
\end{proof}

\section{The associated C$^*$-algebras} \label{SectCStar}

Note that the results in this section are proved independently from Sections \ref{SectStructure} and \ref{SectWeakSimpl}.

The categories $\langle\pi_l,l\leq k\rangle$ for $k\in\N\cup\{\infty\}$ give rise to the \emph{non group-theoretical hyperoctahedral series} of easy quantum groups, denoted by $H_n^{[\pi_k]}$. All these quantum groups are quantum subgroups of the free hyperoctahedral quantum group  $H_n^{[\pi_1]}=H_n^+$.
The quantum group $H_n^{[\pi_2]}$ associated to the category $\langle\pi_2\rangle=\langle\fatcrosspart\rangle$ is also denoted by  $\Hnfc$. We call it the \emph{square commuting (hyperoctahedral) quantum group}. 

We first give a description of the C$^*$-algebras associated to the categories $\langle\pi_l,l\leq k\rangle$ (resp. to the quantum groups $H_n^{[\pi_k]}$) before we prove that they are actually all mutually different.
Let us recall how to translate the fact that a linear map $T$ is in the intertwiner space of a quantum group into relations on the generators $u_{ij}$. If $A$ is a $C^*$-algebra generated by elements $u_{ij}$, $1\leq i,j\leq n$, the map $T:(\C^n)^{\otimes k}\to(\C^n)^{\otimes l}$ gives rise to a matrix $T\otimes 1\in M_{n^k\times n^l}(A)$. On the other hand, $u^{\otimes k}$ is shorthand for the matrix $(u_{i_1j_1}\ldots u_{i_kj_k})$ in $M_{n^k\times n^k}(A)$ which may be seen as mapping  vectors from $(\C^n)^{\otimes k}\otimes A$ to $(\C^n)^{\otimes l}\otimes A$.
 We study the matrices $(T\otimes 1)$ and $ u^{\otimes k}$ by evaluation at the $i$-th entry, where $i=(i_1,\ldots,i_k)$ is a multi-index. Thus:
\begin{align*}
 (T\otimes 1) (e_{i_1}\otimes\ldots\otimes e_{i_k}) &= T(e_{i_1}\otimes\ldots\otimes e_{i_k})\otimes 1\\
 u^{\otimes k}(e_{i_1}\otimes\ldots\otimes e_{i_k}) &= \sum_{\alpha_1,\ldots,\alpha_k} (e_{\alpha_1}\otimes\ldots\otimes e_{\alpha_k})\otimes u_{\alpha_1 i_1}\ldots u_{\alpha_k i_k}
\end{align*}
These methods will be used in the sequel.

\begin{proposition}\label{PropCStarHnfc}
 The C$^*$-algebra associated to the quantum group $\Hnfc$ is the universal unital C$^*$-algebra generated by elements $u_{ij}$ such that the following relations hold:
 \begin{itemize}
  \item[(i)] $u_{ij}^*=u_{ij}$ for all $i,j$
  \item[(ii)] $u_{ij}^2u_{kl}^2=u_{kl}^2u_{ij}^2$ for all $i,j,k,l$
  \item[(iii)] $u_{ik}u_{jk}=u_{ki}u_{kj}=0$ for all $i,j,k$ with $i\neq j$
  \item[(iv)] $\sum_k u_{ik}^2=\sum_k u_{ki}^2=1$ for all $i$
 \end{itemize}
In particular, the elements $u_{ij}$ are partial isometries.
\end{proposition}
\begin{proof}
The proof follows from standard techniques, see for instance \cite{Weber13}. Since the partitions $p=\paarpart$ and $q=\vierpart$ are in $\langle\fatcrosspart\rangle$, we have that the corresponding maps $T_p$ and $T_q$ are in the intertwiner spaces of $\Hnfc$. We deduce that the relations (i), (iii) and (iv) hold in the C$^*$-algebra asociated to $\Hnfc$ (see \cite{Weber13}). 

In order to prove that (ii) holds in $C(\Hnfc)$, note that the linear map $T_p:(\CC^n)^{\otimes 4}\to(\CC^n)^{\otimes 4}$ associated to $p=\;\fatcrosspart$ is of the form:
 \[T_p(e_a\otimes e_b\otimes e_c\otimes e_d)=\delta_{ab}\delta_{cd} e_c\otimes e_c\otimes e_a\otimes e_a\]
 We compute using (iii) (see also \cite{Weber13} for such computations):
 \begin{align*}
  u^{\otimes 4}(T_p\otimes 1)(e_a\otimes e_b\otimes e_c\otimes e_d)
  &= \delta_{ab}\delta_{cd}\sum_{\alpha\beta\gamma\delta} e_\alpha\otimes e_\beta\otimes e_\gamma\otimes e_\delta\otimes u_{\alpha c}u_{\beta c} u_{\gamma a} u_{\delta a}\\
  &= \delta_{ab}\delta_{cd}\sum_{\alpha\gamma} e_\alpha\otimes e_\alpha\otimes e_\gamma\otimes e_\gamma\otimes u_{\alpha c}^2 u_{\gamma a}^2\\
 (T_p\otimes 1)u^{\otimes 4}(e_a\otimes e_b\otimes e_c\otimes e_d)
  &= \sum_{\gamma\delta\alpha\beta} (T_p\otimes 1)e_\gamma\otimes e_\delta\otimes e_\alpha\otimes e_\beta\otimes u_{\gamma a}u_{\delta b} u_{\alpha c} u_{\beta d}\\
  &= \sum_{\gamma\alpha} e_\alpha\otimes e_\alpha\otimes e_\gamma\otimes e_\gamma\otimes u_{\gamma a}u_{\gamma b} u_{\alpha c} u_{\alpha d}\\
  &= \delta_{ab}\delta_{cd}\sum_{\gamma\alpha} e_\alpha\otimes e_\alpha\otimes e_\gamma\otimes e_\gamma\otimes  u_{\gamma a}^2u_{\alpha c}^2   
 \end{align*} 
We conclude that  $u^{\otimes 4}(T_p\otimes 1)=(T_p\otimes 1)u^{\otimes 4}$ if and only if the relations (ii) hold.
 
On the other hand, the universal C$^*$-algebra $A$ generated by the above relations (i), (ii), (iii) and (iv) gives rise to a compact matrix quantum group since all relations are preserved under the comultiplication $\Delta$. It follows that the maps $T_p$ are in the intertwiner spaces of this quantum group, for $p\in\{\paarpart,\vierpart,\fatcrosspart\}$. Thus, the intertwiner spaces contain all maps $T_p$, for $p\in\langle\fatcrosspart\rangle$ and $A$ is a quotient of $C(\Hnfc)$.

 Note that the elements $u_{ij}^2$ are projections due to (iii) and (iv):
\[u_{ij}^2=u_{ij}^2\left(\sum_k u_{ik}^2\right)=\sum_k u_{ij}^2u_{ik}^2=u_{ij}^4\]
Hence, the $u_{ij}$ are partial isometries.
\end{proof}

We now describe the $C^*$-algebras corresponding to the quantum groups $H_n^{[\pi_k]}$, denoted by $A_{\pi_k}(n)$. The preceding proposition is just a special case.

\begin{proposition}\label{PropCStarPik}
 Let $k\in\{2,3,\ldots,\infty\}$. The C$^*$-algebra $A_{\pi_k}(n)$ associated to the category $\langle\pi_l,l\leq k\rangle$ (and to the quantum group $H_n^{[\pi_k]}$) is the universal unital C$^*$-algebra generated by elements $u_{ij}$ such that the following relations hold:
 \begin{itemize}
  \item[(i)] $u_{ij}^*=u_{ij}$ for all $i,j$
  \item[(ii)] For all $2\leq l\leq k$ and all multi-indices $i$ and $j$ of length $l$ the elements $(u_{i_2j_2}\ldots u_{i_lj_l}^2\ldots u_{i_2j_2})$ and $u_{i_1j_1}^2$ commute.
  \item[(iii)] $u_{ik}u_{jk}=u_{ki}u_{kj}=0$ for all $i,j,k$ with $i\neq j$
  \item[(iv)] $\sum_k u_{ik}^2=\sum_k u_{ki}^2=1$ for all $i$
 \end{itemize}
In particular, the elements $u_{ij}$ are partial isometries.
\end{proposition}
\begin{proof}
 For elements $u_{ij}, 1\leq i,j\leq n$ we define the relations (ii)' by:
\begin{itemize}
 \item[(ii)'] For all $2\leq l\leq k$ and all multi-indices $i$ and $j$ of length $l$ we have:
\begin{align*}
&\delta_{ij}(u_{\alpha_2i_2}\ldots u_{\alpha_li_l}u_{\beta_li_l}\ldots u_{\beta_2i_2})(u_{\beta_1i_1}u_{\alpha_1i_1})\\
=\;&\delta_{\alpha\beta}(u_{\alpha_1j_1}u_{\alpha_1i_1})(u_{\alpha_2i_2}\ldots u_{\alpha_li_l}u_{\alpha_lj_l}\ldots u_{\alpha_2j_2})
\end{align*}
\end{itemize}
Here,  $\delta_{ij}$ is shorthand for the product $\delta_{i_1j_1}\ldots \delta_{i_lj_l}$ and likewise $\delta_{\alpha\beta}$.

We first prove that the universal C$^*$-algebra $A$ generated by elements $u_{ij}$ and the relations (i), (ii), (iii) and (iv) is isomorphic to the C$^*$-algebra $B$ generated by (i), (ii)', (iii) and (iv). Putting $i=j$ and $\alpha=\beta$ in (ii)', we immediately infer that $B$ is a quotient of $A$ under the map sending generators to generators. Conversely, the relations (ii)' with $i=j$ and $\alpha=\beta$ follow directly from (ii). For $i\neq j$ and $\alpha\neq \beta$ nothing is to prove. So let $i=j$ and $\alpha\neq \beta$. We have to prove that 
\[(u_{\alpha_2i_2}\ldots u_{\alpha_li_l}u_{\beta_li_l}\ldots u_{\beta_2i_2})(u_{\beta_1i_1}u_{\alpha_1i_1})=0\]
holds in $A$. If $\alpha_1\neq\beta_1$ or $\alpha_l\neq\beta_l$, the above equation is true due to (iii). Otherwise, there is an index $1<s<l$ such that $\alpha_s\neq\beta_s$ and $\alpha_t=\beta_t$ for all $s<t\leq l$. Since $u_{\beta_si_s}$ is a partial isometry and using the relations (ii), we infer:
\begin{align*}
&u_{\alpha_si_s}(u_{\alpha_{s+1}i_{s+1}}\ldots u_{\alpha_li_l}u_{\beta_li_l}\ldots u_{\beta_{s+1}i_{s+1}})u_{\beta_si_s}\\
&=u_{\alpha_si_s}(u_{\alpha_{s+1}i_{s+1}}\ldots u_{\alpha_li_l}^2\ldots u_{\alpha_{s+1}i_{s+1}})u_{\beta_si_s}\\
&=u_{\alpha_si_s}(u_{\alpha_{s+1}i_{s+1}}\ldots u_{\alpha_li_l}^2\ldots u_{\alpha_{s+1}i_{s+1}})u_{\beta_si_s}^3\\
&=u_{\alpha_si_s}u_{\beta_si_s}^2(u_{\alpha_{s+1}i_{s+1}}\ldots u_{\alpha_li_l}^2\ldots u_{\alpha_{s+1}i_{s+1}})u_{\beta_si_s}\\
&=0
\end{align*}
The last step follows from (iii) which yields $u_{\alpha_si_s}u_{\beta_si_s}=0$.

We claim that the relations (ii)' are equivalent to the fact that the following auxilary partitions $\pi_l'$ are in the intertwiner spaces, for $2\leq l\leq k$.
\begin{align*}
\pi_l' &:= \begin{pmatrix} 
a_1 &a_1 &a_2 &\ldots &a_l &a_l &\ldots &a_2\\
a_2 &\ldots &a_l &a_l &\ldots &a_2 &a_1 &a_1 \end{pmatrix}
\end{align*}
This partition is a rotated version of $\pi_l$. Thus, $\pi_l\in\CC$ if and only if    $\pi'_l\in\CC$. We prove that $T_{\pi_l'}u^{\otimes 2l}=u^{\otimes 2l}T_{\pi_l'}$ if and only if the relations (ii)' hold.
The map $T_{\pi_l'}$ satisfies:
 \[T_{\pi_l'}((e_{j_1}\otimes e_{i_1})\otimes e_{i_2}\otimes \ldots\otimes e_{i_l}\otimes e_{j_l}\otimes \ldots\otimes e_{j_2})=\delta_{ij}e_{i_2}\otimes \ldots\otimes e_{i_l}\otimes e_{i_l}\otimes \ldots\otimes e_{i_2}\otimes(e_{i_1}\otimes e_{i_1}) \]
 We then compute:
 \begin{align*}
  u^{\otimes 2l}(T_{\pi_l'}\otimes 1)&(((e_{j_1}\otimes e_{i_1})\otimes e_{i_2}\otimes \ldots\otimes e_{i_l}\otimes e_{j_l}\otimes \ldots\otimes e_{j_2})\otimes 1)\\
  &= \delta_{ij}\sum_{\alpha\beta} e_{\alpha_2}\otimes \ldots\otimes e_{\alpha_l}\otimes e_{\beta_l}\otimes \ldots\otimes e_{\beta_2}\otimes e_{\beta_1}\otimes e_{\alpha_1} \\
 &\qquad\qquad\otimes(u_{\alpha_2 i_2}\ldots u_{\alpha_l i_l}u_{\beta_l i_l}\ldots u_{\beta_2 i_2})(u_{\beta_1 i_1}u_{\alpha_1 i_1})\\
  (T_{\pi_l'}\otimes 1)u^{\otimes 2l}&((e_{j_1}\otimes e_{i_1})\otimes e_{i_2}\otimes \ldots\otimes e_{i_l}\otimes e_{j_l}\otimes \ldots\otimes e_{j_2})\otimes 1)\\
  &= \sum_{\alpha\beta} (T_{\pi_l'}\otimes 1)e_{\beta_1}\otimes e_{\alpha_1}\otimes e_{\alpha_2}\otimes \ldots\otimes e_{\alpha_l}\otimes e_{\beta_l}\otimes \ldots\otimes e_{\beta_2}\\
  &\qquad\qquad\otimes (u_{\beta_1 j_1}u_{\alpha_1 i_1})  (u_{\alpha_2 i_2}\ldots u_{\alpha_l i_l}u_{\beta_l j_l}\ldots u_{\beta_2 j_2})\\
  &= \sum_{\alpha}  e_{\alpha_2}\otimes \ldots\otimes e_{\alpha_l}\otimes e_{\alpha_l}\otimes \ldots\otimes e_{\alpha_2}\otimes (e_{\alpha_1}\otimes e_{\alpha_1})\\
  &\qquad\qquad\otimes (u_{\alpha_1 j_1}u_{\alpha_1 i_1})  (u_{\alpha_2 i_2}\ldots u_{\alpha_l i_l}u_{\alpha_l j_l}\ldots u_{\alpha_2 j_2})
 \end{align*} 
Comparison of the coefficients yields the claim.
\end{proof}

It remains to show, that the inclusions
\[\langle\fatcrosspart\rangle\subset\langle\pi_l,l\leq 3\rangle
\subset\langle\pi_l,l\leq 4\rangle\subset\ldots\subset\langle \pi_k, k\in\N\rangle\subset \langle\primarypart\rangle\]

are strict. In general, if $\CC\subset\mathcal D$ are two categories with associated quantum groups $G(\CC)$ and $G(\mathcal D)$ and associated C$^*$-algebras $A_\CC(n)$ and $A_{\mathcal D}(n)$, we have a surjective homomorphism $\phi:A_\CC(n)\to A_{\mathcal D}(n)$, mapping generators to generators. Thus, if we find a representation $\sigma:A_\CC(n)\to B(H)$ on some Hilbert space $H$ such that the elements $\sigma(u_{ij})$ do \emph{not} fulfill the relations of $A_{\mathcal D}(n)$, the homomorphism $\phi$ cannot be an isomorphism. This proves $\CC\neq\mathcal D$.

\begin{proposition}\label{PiKDifferent1}
 Let $k\in\N$, $k\geq 2$. There is a dimension  $n_k\in\N$, a representation $\sigma_k: A_{\pi_k}(n_k)\to B(H_k)$ and multi-indices $i$ and $j$ of length $k+1$ such that the elements $\sigma_k(u_{i_2j_2}\ldots u_{i_{k+1}j_{k+1}}^2\ldots u_{i_2j_2})$ and $\sigma_k(u_{i_1j_1}^2)$ do \emph{not} commute.
In particular, the categories $\langle\pi_l,l\leq k\rangle$ for $k\in\{2,3,\ldots,\infty\}$ are mutually different.
\end{proposition}
\begin{proof}
Consider $H:=\C^2\otimes \C^k$ and let $p_1,\ldots,p_k$ be the projections onto the subspaces $\C e_1,\ldots,\C e_k$ of $\C^k$ respectively. Here, $e_1,\ldots,e_k$ is the canonical basis of $\C^k$. For $i=1,\ldots, k-1$, let $\hat v_i\in B(\C^k)$ be the operator swapping the basis vectors $e_i$ and $e_{i+1}$ and leaving all other basis vectors invariant. That is, we define:
\[\hat v_i e_j:=
\begin{cases} 
 e_{i+1} &\textnormal{ if $j=i$}\\
 e_i &\textnormal{ if $j=i+1$}\\
 e_j &\textnormal{ otherwise}            
\end{cases}\]
Furthermore, let $\mathring p$ and $\mathring q$ be two noncommuting projections in $B(\C^2)$, for instance:
\[\mathring p:=\begin{pmatrix} 1 & 0\\ 0& 0\end{pmatrix} \qquad \textnormal{and} \qquad 
\mathring q:=\frac{1}{2}\begin{pmatrix} 1 & 1\\ 1& 1\end{pmatrix}\]
We define operators in $B(H)$ by the following:
\[v_i:= 1\otimes \hat v_i, \textnormal{ for } i=1,\ldots,k-1 \qquad p:=\mathring p\otimes p_1,\qquad q:=\mathring q\otimes p_k\]
Note that the projections $p$ and $q$ commute (they are orthogonal) whereas $\mathring p$ and $\mathring q$ do not. Furthermore, we have $v_i=v_i^*$, $v_i^2=1$ and $v_i\ldots v_1pv_1\ldots v_i=\mathring p\otimes p_{i+1}$.
We show that for $n:=k+3$ the following matrix $(u'_{ij})_{1\leq i,j\leq n}$ gives rise to a representation $\sigma_k$ of $A_{\pi_k}(n)$ on $H$, mapping $u_{ij}$ to $u_{ij}'$.
\[(u'_{ij})_{1\leq i,j\leq n}:= 
\begin{pmatrix}
v_1\\
& \ddots\\
& & v_{k-1}\\
& & &p &1-p\\
& & &1-p &p\\
& & & & &q &1-q\\
& & & & &1-q &q
\end{pmatrix}\]
Here, all other entries are zero. We immediately see that the elements $u_{ij}'$ fulfill (i), (iii) and (iv) of the relations of Proposition \ref{PropCStarPik}. As for (ii), let $2\leq l\leq k$. We prove by induction on $l$ that the elements $(u'_{i_2j_2}\ldots u_{i_lj_l}'^2\ldots u'_{i_2j_2})$ and $u_{i_1j_1}'^2$ commute for all choices of elements $u_{i_sj_s}'$. The claim is true for $l=2$. For the induction step, we may assume that all elements $u_{i_sj_s}'$ are nonzero and $u_{i_1j_1}'$ and $u_{i_lj_l}'$ are in $\{p,1-p,q,1-q\}$ (for the assumption on $u_{i_lj_l}'$ we use the induction hypothesis).

\textbf{\emph{Case 1.} The elements $u_{i_2j_2}',\ldots,u_{i_{l-1}j_{l-1}}'$ are all in $\{v_1,\ldots,v_{k-1}\}$.}

If $u_{i_lj_l}'=p$, the element $(u'_{i_2j_2}\ldots u_{i_lj_l}'^2\ldots u'_{i_2j_2})$ is of the form $\mathring p\otimes p_r$ for some $r$ with  $1\leq r\leq l-1\leq k-1$. Hence it commutes with $p$ and $q$, thus also with $u_{i_1j_1}'^2\in\{p,1-p,q,1-q\}$. If $u_{i_lj_l}'=q$, we have $(u'_{i_2j_2}\ldots u_{i_lj_l}'^2\ldots u'_{i_2j_2})=\mathring q\otimes p_r$ for some $r$ with  $2\leq k-(l-2)\leq r\leq k$. Hence it commutes with $u_{i_1j_1}'^2$. This also proves the cases $u_{i_lj_l}'=1-p$ and $u_{i_lj_l}'=1-q$.

\textbf{\emph{Case 2.} There is a $u_{i_sj_s}'$ for some $2\leq s\leq l-1$ such that $u_{i_sj_s}'\in\{p,1-p,q,1-q\}$.}

In this case, the element $(u'_{i_2j_2}\ldots u_{i_lj_l}'^2\ldots u'_{i_2j_2})$ is of the form:
\[(u'_{i_2j_2}\ldots u_{i_lj_l}'^2\ldots u'_{i_2j_2})=XaYbY^*aX^*\]
Here, $X$ and $Y$ are some products of elements $u_{i_sj_s}'$ and the word $Y$ has length less than $l-2$. Furthermore $a,b\in\{p,1-p,q,1-q\}$. Like in case 1, we deduce that $aYbY^*a$ is in $\{p,q,YbY^*,0\}$ if $b\in\{p,q\}$. Hence we can apply the induction hypothesis. The same holds for the case $b\in\{1-p,1-q\}$.

We have shown that (ii) of Proposition \ref{PropCStarPik} is fulfilled, hence the assignment $u_{ij}\mapsto u_{ij}'$ is a representation of $A_{\pi_k}(n)$. Now, let $u_{i_1j_1}':=q=\mathring q\otimes p_k$ and put $u_{i_2j_2}',\ldots, u_{i_{k+1}j_{k+1}}'$. Then:
\[u_{i_2j_2}'\ldots u_{i_kj_k}'u_{i_{k+1}j_{k+1}}'^2u_{i_kj_k}'\ldots u_{i_2j_2}'=v_{k-1}\ldots v_2v_1pv_1v_2\ldots v_{k-1}=\mathring p\otimes p_k\]
So the elements $u_{i_2j_2}'\ldots u_{i_{k+1}j_{k+1}}'^2\ldots u_{i_2j_2}'$ and $u_{i_1j_1}'$ do \emph{not} commute, since $\mathring p$ and $\mathring q$ do not commute.

Now, consider the surjections $A_{\pi_k}(n)\to A_{\pi_{k+1}}(n)$ mapping generators to generators. Since $\sigma_k$ is a representation of $A_{\pi_k}(n)$ but not of $A_{\pi_{k+1}}(n)$, the categories  $\langle\pi_l,l\leq k\rangle$ and $\langle\pi_l,l\leq k+1\rangle$ do not coincide, for all $k\in\N, k\geq 2$.  
It follows that $\langle\pi_l,l\leq k\rangle\neq\langle\pi_l,l\leq \infty\rangle$ for all $k\in\N, k\geq 2$. 
\end{proof}

\begin{remark}\label{PiKDifferent2}
 The following representation $\sigma_{\infty}:A_{\pi_\infty}(n)\to B(H_\infty)$ shows that $\langle\pi_l,l\leq \infty\rangle\neq \langle\primarypart\rangle$. Let $H_\infty:= \C^2$ and put
\[p:=\begin{pmatrix} 1 & 0\\ 0& 0\end{pmatrix} \qquad \textnormal{and} \qquad 
w:=\begin{pmatrix} 0 & 1\\ 1& 0\end{pmatrix}\;.\]
Then the assignment
\[(u_{ij})\mapsto \begin{pmatrix} p & 1-p\\ 1-p&p\\0& 0&w\end{pmatrix}\]
is a representation of $A_{\pi_\infty}(3)$ where the images of $u_{11}^2$ and $u_{33}$ do \emph{not} commute (but this is the case in the C$^*$-algebra associated to $\langle\primarypart\rangle$, see also \cite{RaumWeberSimpl}). Thus, all $\langle\pi_l,l\leq k\rangle$ are non group-theoretical.
\end{remark}

\begin{remark}\label{PiKDifferent3}
Since $\langle\primarypart\rangle\subset\langle\vierpart,\crosspart\rangle$ and all partitions in  $\langle\vierpart,\crosspart\rangle$  have blocks of even size (see \cite{BanicaSpeicher09}), we infer that
 $\singleton\otimes\singleton\notin\langle\primarypart\rangle$.
This also proves that $\singleton\otimes\singleton$ is not contained in $\langle\pi_l,l\leq k\rangle$, hence these categories are indeed non group-theoretical and hyperoctahedral.
\end{remark}

%

\end{document}